\documentclass[12pt,a4paper,reqno]{amsart}
\usepackage{amsmath,amssymb,amsthm,amscd,mathrsfs} 
\usepackage[T1]{fontenc}
\usepackage[latin1]{inputenc}
\usepackage[english]{babel}
\usepackage[mathscr]{eucal}
\usepackage{graphicx}
\usepackage{float}
\usepackage{tikz}
\usetikzlibrary{positioning}
\usetikzlibrary{shapes.geometric}
\usetikzlibrary{fit}
\usetikzlibrary{patterns}
\usepackage{caption}
\usepackage{cases}
\usepackage{hyperref}
\usepackage{enumitem}
\usetikzlibrary{decorations.text}

\input{comment.sty}
\includecomment{FZ}
\includecomment{MM}
\counterwithin{equation}{section}

\renewcommand{\b}{\beta}

\renewcommand{\th}{\theta}

\renewcommand{\l}{\lambda}

\newcommand{\n}{\nu}

\newcommand{\ph}{\phi}
\def\ph{\phi}
\def\md#1{\ \mbox{\rm(mod }{#1})}

\def\npp#1{N_{\ph}^+(#1)}
\def\ol{\overline}

\def\ph{\phi}

\newcommand{\Q}{{\mathbb Q}}
\newcommand{\Z}{{\mathbb Z}}
\newcommand{\N}{{\mathbb N}}
\newcommand{\F}{{\mathbb F}}

\def\md#1{\ \mbox{\rm(mod }{#1})}

\def\npp#1{N_{\ph}^+(#1)}

\newcommand{\aF}{\mathfrak a}

\newcommand{\pF}{\mathfrak p}

\newtheorem{theorem}{Theorem}[section]

\newtheorem{lemma}[theorem]{Lemma}
\newtheorem{corollary}[theorem]{Corollary}

\theoremstyle{definition}

\newtheorem{definitions}[theorem]{Definitions}

\theoremstyle{remark}
\newtheorem{example}[theorem]{Example}

\newtheorem{remarks}[theorem]{Remarks}
\newtheorem{examples}[theorem]{Examples}

\begin{document}
\title[]{On common index divisors and monogenity of  of the  nonic  number field defined by a trinomial   $x^9+ax+b$}
\textcolor[rgb]{1.00,0.00,0.00}{}
\author{  Hamid Ben Yakkou and  Pagdame Tiebekabe }\textcolor[rgb]{1.00,0.00,0.00}{}
\address{Faculty of Sciences Dhar El Mahraz, P.O. Box  1874 Atlas-Fes , Sidi mohamed ben Abdellah University,  Morocco}\email{beyakouhamid@gmail.com }
\address{University of Kara, Faculty of Sciences and techniques,   Department of Mathematics, P.O.BoX 43 Kara, Togo}\email{pagdame.tiebekabe@ucad.edu.sn}
\maketitle
\vspace{0.3cm}
\begin{abstract}  
	Let $K $ be a nonic number field generated by   a complex   root $\th$ of a monic irreducible trinomial  $ F(x)= x^9+ax+b \in \Z[x]$, where $ab \neq 0$. Let $i(K)$ be the index of $K$.  A rational prime  $p$ dividing $ i(K)$ is called a prime common index divisor of $K$.  In this paper, for every rational  prime $p$, we give necessary and sufficient  conditions depending  only  $a$ and $b$ for which $p$ is a common index divisor of $K$. As application of our results we identify infinite parametric families  of non-monogenic nonic numbers fields defined by such trinomials.   At the end,	some numerical examples illustrating our theoretical results are given.
\end{abstract}
\maketitle
\section{Introduction and statements of results}
	Let $K $ be a number field generated by $\th$, a root of a monic irreducible polynomial $F(x)\in \Z[x]$ of degree $n$, and $\Z_K$  its ring of integers.  By \cite[Theorem 2.10]{Na}, the field $K$ has  at least one integral basis; $\Z_K$ is a free $\Z$-module of rank $n$.   For every $\eta \in \Z_K$ generating $K$, let  $(\Z_K: \Z[\eta])= |\Z_K/\Z[\eta]|$  be the index of $\Z[\eta] $ in $\Z_K$,  called the index of $\eta$. The ring $\Z_K$ is said to have a power integral basis if it admits  a $\Z$-basis $(1,\eta,\ldots,\eta^{n-1})$ for some $\eta\in \Z_K$; $\Z_K=\Z[\eta]$.  In a such case, the number field $K$ is said to be monogenic. Otherwise, $K$ is said not monogenic.
	Let $(1, \omega_2, \ldots, \omega_n)$ be an arbitrary integral basis of $K$. Then the discriminant of the linear form  $$l(x)= x_1+ x_2 \omega_2+ \cdots + x_n \omega_n$$ is equal to 
	\begin{eqnarray}\label{discdisc}
	\prod_{1 \le i < j \le n} (l^{(j)}(x)- l^{(i)}(x))^2 = I(x_2, \ldots, x_n)^2 \cdot D_K,
	\end{eqnarray} where $l^{(j)}(x), 1 \le j \le n$ the conjugates of $l(x)=l^{(1)}(x)$,   $x=(x_1, x_2, \ldots, x_n) \in \Z^n$,  $D_K$ is the discriminant of $K$,  and $I(x_2, \ldots, x_n)$  is a homogeneous form of degree $\frac{n(n-1)}{2}$ 
	in $n-1$ variables with coefficients in $\Z$ called the index form corresponding to the integral basis  $(1, \omega_2, \ldots, \omega_n)$ (see \cite[Lemma 1.4]{G19}).  Recall also that for any $\eta \in \Z_K$, it is well known by \cite[Proposition 2.13]{Na} that:   
	\begin{eqnarray}\label{indexdiscrininant}
	D(\th) = ( \Z_K :  \Z [\eta])^2 \cdot D_K,
	\end{eqnarray}
	where $D(\eta)$ is the discriminant of the minimal polynomial of $\th$.
	Representing any $\eta$ of $ \Z_K$ in the form $\th = x_1+x_2 \omega_2+ \cdots+x_n \omega_n$, and combining (\ref{discdisc}) and (\ref{indexdiscrininant}),  we have $ ( \Z_K : \Z [\eta]) = |I(x_2, \ldots, x_n)| $ (see \cite[Lemma 1.5]{G19}). So, we have to solve the following equation to decide about the monogenity of $K$, and to find all generators of power integral bases:
	\begin{eqnarray}\label{indexform}
	I(x_2, \ldots, x_n) = \pm 1 \,\,
	( \text{in}\,\, x_2, x_3, \ldots, x_n \in \Z). 
	\end{eqnarray}In this paper, $i(K)$ will denote  the  index of   $K$, defined as follows: 
\begin{equation}\label{Defi(K)}
 i(K) = \gcd \ \{(Z_K:\Z[\eta]) \,| \ \eta \in \mathbb{Z}_K  \, \mbox{and}\,  K=\Q(\eta)\}. 
\end{equation} 
A rational prime $p$ dividing $i(K)$ is called a common index divisor of $K$ (or nonessential discriminant divisor of $K$). Note that if $K$ is monogenic, for $i(K) = 1$. But, if $i(K)>1$; equivalently if there exist a rational  prime $p$ dividing $i(K)$,  then Equation $(\ref{indexform})$ has no integral solutions and so $K$ is not monogenic.  The canonical examples of monogenic number fields are quadratic and cyclotomic fields. Indeed, any quadratic field $K=\Q(\sqrt{d})$, with $d\in \Z$ and square free, has ring of integers given by: $$\Z_K=
\left
\{\begin{array}{ll} \Z[\sqrt{d}],& \mbox{if } d \equiv 2, 3 \md 4,\\
\Z[\frac{1+\sqrt{d}}{2}],&\mbox{if } d \equiv 1 \md 4.
\end{array}
\right.$$ Also, if $K=\Q(\xi_n)$ is any cyclotomic number field, where $n \in \N$ and $\xi_n$ is a primitive  nth root of unity, then $\Z_K = \Z[\xi_n]$.   The first non-monogenic number field was given by Dedekind in 1878. He showed that the cubic number field $\Q(\th)$ is not monogenic when $\th$ is a root of the polynomial $x^3-x^2-2x-8$ (cf. \cite[p.64]{Na}). More precisely, he showed that $2$ is a common index divisor of $K$ since it is split completely in $\Z_K$.  The problem of testing the monogenity of number fields and constructing power integral bases have been intensively studied. There are extensive   results regarding this problem. These results were  treated by using  different approachs. Ga\'{a}l, Gy\H{o}ry,  Peth\H{o}, Pohst,  Remete (cf. \cite{Gacta2, G19, Gacta1,  GRN, GR17,Gyoryrsurlespolynomes,Gyoryredecide, PP}) and their  research teams who succeeded to study the monogenity of several number fields, are  based in their approach on the  resolution of  index form equations. It was proved by   Gy\H{o}ry that   the index form equation  can have only many integral solutions for which he  gave effective bounds (see \cite{Gyoryrsurlespolynomes,Gyoryredecide,Gyoryrdiscriminant,Gyoryr1983}). In \cite{GP}, Ga\'{a}l,  Peth\H{o} and Pohst studied indices  of quartic number fields. As consequence of their results,  they computed minimal indices and all elements of minimal index in all totally real quartic number fields with Galois group $A_4$ and discriminant fields $<10^6$.   In \cite{DS}, Davis and Spearman studied the index of the quartic number  field defined by $x^4+ax+b$.   In\cite{Gacta1}, Ga\'{a}l and Gy\H{o}ry described an algorithm to solve  index form equations in quintic number fields  and they computed all generators of power integral bases in some totally real quintic number fields with Galois group $S_5$. In  \cite{Gacta2}, Bilu, Ga\'{a}l and Gy\H{o}ry   studied  the monogenity of  some  totally real sextic number fields with Galois group $S_6$. In \cite{GR17},  Ga\'al and Remete studied the monogenity of pure number  fields  $\Q(\sqrt[n]{m}),$ where  $3\leq n\leq 9$ and $m$ is square free. They  also showed in \cite{GRN} that for a square free rational integer  $m \equiv 2, 3  \md{4}$,  the octic number filed $\Q(i, \sqrt[4]{m})$ is not monogenic.   In \cite{ PP},   Peth\H{o} and  Pohst studied indices in multiquadratic number fields.  Also, in \cite{PethoZigler}, Peth\H{o} and Ziegler gave an efficient criterion to decide whether the maximal order of a biquadratic field has a unit power integral basis or not. For important theoretical results about the  monogenity of relative extensions, see  \cite{Gyoryrelative} by Gy\H{o}ry.     The books \cite{EG} by Evertse and   Gy\H{o}ry, and \cite{G19} by  Ga\'al  give detailed surveys on the discriminant, the index form theory and its applications, including related Diophantine equations and monogenity of number fields. Nakahara's research team   based on the existence of relative  power  integral bases of some special  sub-fields, they studied the monogenity of several number fields:   Ahmad,  Nakahara and  Husnine \cite{ANHN} proved that for a square free rational integer $m$, if $m \equiv 2, 3 \md 4$ and $m \not \equiv  \pm 1 \md 9, $ then the sextic pure field $K = \Q(\sqrt[6]{m})$ is monogenic. But it is not monogenic when $m\equiv 1\md{4}$, see \cite{AN} by Ahmad,  Nakahara and  Hameed. In \cite{Smtacta},  Smith studied the monogenity of radical extensions and  he  gave  sufficient conditions for a  Kummer extension to be not monogenic.  Recall that the polynomial $F(x)$ is said to be monogenic if $\Z_K = \Z[\th]$; that is  $(1, \th, \ldots, \th^{n-1})$ form  a power integral  basis  of $\Z_K$.  Note that the monogenity of  the polynomial $F(x)$ implies the monogenity of the field $K$. But, the converse is not true, because a number field defined by non-monogenic polynomial can be monogenic;  the non-monogenity of $F(x)$ does not imply the non-monogenity of $K$ (see Theorem 2.1  of  \cite{BFT1} by Ben Yakkou and El Fadil,  which gives infinite parametric monogenic number field defined by non-monogenic trinomials). Let us recall some result concerning the monogenity of trinomials.
In \cite{JW}, Jones and White gave classes  of monogenic trinomials
 with non square free discriminant. Also, in \cite{JP}, Jones and Phillips construct infinite families of  monogenic trinomials defined by $F_n(x)=x^n+a(m,n)x+b(m,n)$ where $a(m,n)x$ and $b(m,n)$ are certain prescribed forms in $m$. 
Recently, in 2021,   Ga\'al studied the monogenity of sextic number fields defined by $x^6+ax^3+b$ (see \cite{Ga21}). Based on the important works of \O. Ore, Gu\`{a}rdia,  Montes and  Nart about the application of   Newton polygon techniques  on the factorization of ideals of the ring $\Z_K$ into a product of powers of prime ideals (see \cite{EMN, Nar, Narprime, MN92, O}), several results obtained about the monogenity and indices of number fields defined by trinomials.
 For $x^5+a^3+b$ and $x^8+ax+b$, see respectively  \cite{Barxiv} and \cite{BATCA} by  Ben Yakkou. For $x^7+ax^3+b$, see  \cite{FKcom} by  El Fadil and Kchit. The goal of the present paper is to study  the index  of the  number field generated by a root  of a monic irreducible trinomial of type  $F(x)=x^9+ax+b$.  Recall that  in \cite{BFT1} Ben Yakkou and El Fadil studied the non-monogenity of number fields defined by $x^n+ax+b$. Their results are extended in \cite{BRM} by Ben Yakkou for number fields defined by $x^n+ax^m+b$. 
In what follows, let  $K $ be a number field generated by $\th$, a root of a monic irreducible trinomial $F(x)= x^9+ax+b \in \Z[x]$, where 
$ab \neq 0$ and $\Z_K$  its ring of integers. For every rational prime 
$p$ and any non-zero  $p$-adic integer $m$, $\n_p(m)$  denote the $p$-adic valuation of $m$; the highest power of $p$ dividing $m$, and $m_p := \frac{m}{p^{\n_p(m)}}$. Without loss of generality, for every prime $p$, we assume that \begin{eqnarray}\label{Hypothese}
\n_p(a)<8 \,\, \text{or} \,\, \n_p(b)<9.
\end{eqnarray}
To explain this,  suppose that $\n_p(a) \ge 8$ and $\n_p(b)\ge 9$. Let $\n_p(a)=8q_1+r_1$ and $\n_p(b)=9q_2+r_2,$ where $1 \le r_1 \le 7$ and $1 \le r_2 \le 8$. Let $q=\min\{q_1, q_2\}, \eta=\frac{\th}{p^q}, A=\frac{a}{p^{8q}}, B=\frac{b}{9^q},$ and $G(x)=x^9+Ax+B$. Then, the following hold:
\begin{enumerate}
\item  $\n_p(A)<8$  or $\n_p(B)<9$
\item $G(x)$ irreducible over $\Q$ and $G(\eta)=0$.
\item $K=\Q(\th)=\Q(\eta)$.
\end{enumerate}So, up to replace $F(x)$ by $G(x)$, the claim holds.\\
For the simplicity of notations, if $p\Z_K=\pF_1^{e_1}\cdots\pF_g^{e_g}$ is the   factorization of $p\Z_K$ into a product of powers of prime ideals in $\Z_K$ with residue degrees $f(\pF_i/p)=[\Z_K/\pF_i : \Z/p\Z]=f_i$, then we write $p\Z_K = [f_1^{e_1}, \ldots, f_g^{e_g}]$. It is also important  to  recall  that   the Fundamental Equality (see \cite[Theorem 4.8.5]{Co}) shows  that 
\begin{eqnarray}\label{FE}
\sum_{i=1}^{g}e_if_i=9=\deg(K).
\end{eqnarray}
 For four  integers $a, b, c$ and $ d$, by the notation $(a, b) \equiv   (c, d) \md{p}$, we mean $a \equiv c \md{p}$ and $b \equiv d \md{p}$. Also, $(a, b) \in S \md{p}$ means that $(a, b)$ equivalent  some element of $S$ modulo $p$.\\
Now, let us state our first main result which gives necessary and sufficient conditions for the divisibility of the index  $i(K)$ by $2$.
\begin{theorem}\label{p=2}   Let $a$ and $b$ be  two rational integers  such that $F(x)=x^9+ax+b \in \Z[x]$ is irreducible over $\Q$, and $K$  a number field generated by a complex root $\th$ of $F(x)$.	Then the form of the factorization of the ideal  $2\Z_K$ into a product of powers of prime ideals of $\Z_K$ is given in Table \ref{table1}. Furthermore, $2$ is common index divisor of $K$ if and only if one of the conditions $A4, A6, A8, A9, A10, A15, A17, A20$ hold.
\end{theorem}
\begin{table}[h!]
	\centering
	\begin{tabular} { | c | c | c |  }
		\hline
		Case & Conditions & Factorization of $2\Z_K$  \\
		\hline
		A1 & $ab \equiv 1 \md{2}$ & $[9^1]$  \\
		\hline
		A2 & $a \equiv 0 \md{2}$ and $b \equiv 1 \md{2}$ & $[1, 2, 6]$  \\
		\hline
		A3 &   $(a,b) \in \{(1, 0), (3, 2)\} \md{4}$ & $[1, 1^8]$  \\
		
		\hline
		A4 &   $a \equiv 1 \md{4}$ and $b \equiv 2 \md{4}$ & $[1, 1, 1^7]$  \\
		\hline
		A5 &   $(a,b) \in \{(3, 0), (7, 4)\} \md{8}$ & $[1, 2^4]$  \\
		\hline
		A6 &   $a \equiv 3 \md{8}$ and $b \equiv 4 \md{4}$ & $[1, 1, 1^3,  1^4]$  \\
		
		\hline
		A7 &   $(a,b) \in \{(7, 0), (15, 8)\} \md{16}$ & $[1, 2^2, 1^4]$  \\
		\hline
		A8 &   $(a,b) \in \{(15, 0), (31, 16)\} \md{32}$ & $[1, 2, 1^2,  1^4]$  \\
		\hline
		A9 &   $(a,b) \in \{(15, 16), (31, 0)\} \md{32}$ & $[1, 1, 1,  1^2,  1^4]$  \\
		\hline
		A10 &   $a \equiv 7 \md{16}$ and $b \equiv 8 \md{16}$ & $[1, 1, 1, 1^2, 1^4], [1, 1^2, 1^2, 1^4],$    \\
		&   & or $[1, 2, 1^2, 1^4]$  \\
		\hline
		A11 & $\n_2(b) \in\{1, 2, 4, 5, 7, 8\}$ and  $9
		\n_2(a)>8\n_2(b)$ & $[1^9]$  \\
		\hline
		A12 & $\n_2(b) \in \{3, 6\}$ and  $9 \n_2(a)>8\n_2(b)$ & $[1^3, 2^3]$  \\
		\hline
		A13 &   $\n_2(a) \in\{1, 3,  5, 7\}$ and  $9 \n_2(a)<8\n_2(b)$ & $[1, 1^8]$  \\
		\hline
		A14 &   $\n_2(a)=2$ and  $\n_2(b) \in \{3, 4\}$ & $[1, 1^8]$  \\
		\hline
		A15 &   $\n_2(a)=2$ and  $\n_2(b)\ge 5$ & $[1, 1^4, 1^4]$  \\
		\hline
		A16 &   $\n_2(a)=6$ and  $\n_2(b)\in \{7, 8\}$ & $[1, 1^8]$  \\
		\hline
		A17 &   $\n_2(a)=6$ and  $\n_2(b)\ge 9$ & $[1, 1^4, 1^4]$  \\
		\hline
		A18 &   $\n_2(a)=4$ and  $\n_2(b)\in \{5, 6\}$ & $[1, 1^8]$  \\
		\hline
		A19 &   $\n_2(a)=4$ and  $\n_2(b)=7$ & $[1^9]$  \\
		\hline
		A20 &   $\n_2(a)=4$ and  $\n_2(b) \ge 8$ & $[1, 1^2, 1^2, 1^2, 1^2], [1, 1^2, 1^2, 1^4],$  \\
		&    & or $[1, 1^2, 1^2, 2^4]$  \\
		\hline
	\end{tabular}
	\caption{Factorization of $2\Z_K$}
	\label{table1}
\end{table}
The following theorem gives necessary and sufficient conditions for which $3$ divides $i(K)$.
\begin{theorem}\label{p=3} Let $K=\Q(\th)$ be a number field with $\th$ a root of a monic irreducible polynomial $F(x)=x^9+ax+b$.  Let $\n=\n_3(1+a+b),$ $\mu=\n_3(9+a)$ and $\omega= \n_3(-1-a+b)$. 
	 Then the form of the factorization of the ideal  $3\Z_K$ into a product of powers of prime ideals of $\Z_K$ is given in Table \ref{table3}.  Furthermore, $3$ is common index divisor of $K$ if and only if one of the conditions  $C28, C30, C32, C42, C46, C48$ hold.
	\end{theorem}
\begin{theorem}\label{p5} Let $K=\Q(\th)$ be a number field with $\th$ a root of a monic irreducible polynomial $F(x)=x^9+ax+b$.	If $p$ is a rational prime greater or equals 5,
	then $p$ is not  a  common index divisor of $K$; $p$ does not divide $i(K)$. 
\end{theorem}
\begin{corollary}\
	Let $K=\Q(\th)$ be a number field with $\th$ a root of a monic irreducible polynomial $F(x)=x^9+ax+b$.   Then
	\begin{enumerate}
		\item   $i(K)>1$ if and only if one of the conditions $A4, A6, A8, A9, A10, A15, A17, A20,  C28, \\ C30, C32, C42, C46, C48$ hold. Otherwise, $i(K)=1$. 
		\item If any of the  conditions  $A4, A6, A8, A9, A10, A15, A17, A20,  C28, C30, C32, C42, C46, \\ C48$ holds, then $K$ is not monogenic; $\Z_K$ has no power integral basis. 
	\end{enumerate}
\end{corollary}
\begin{remarks}\
	The condition $i(K)=1$ is not sufficient for $K$ to be monogenic. Indeed there exist non-monogenic number fields their index equal $1$. For example, consider the pure cubic number field  $K=\Q(\sqrt[3]{m}),$ where $m=1+9k$ and $k=10$ or $12$. In the cubic case the only prime which can divide $i(K)$ is $2$ (see \cite{Engstrom,Zylinski}). Note that  the  minimal polynomial of $\sqrt[3]{m}$ is $F(x)=x^3-m$.  Reducing modulo $2$, we have $F(x) \equiv (x-1)(x^2+x+1) \md{2}$. By using Dedekind's criterion (see \cite[Theorem 6.1.4]{Co} and \cite{R}), $2$ does note divide the index  $(\Z_K : \Z[\sqrt[3]{m}])$. It follows that $i(K)=1$. On the other hand, according to the results of   \cite{GR17},  by Ga\'al and Remete, the index form equation in this case is $I(x_2, x_3)= 3x_1^2+3x_1^2x_2+x_1x_2^2-kx_2^3 \pm 1$ and is not solvable. Hence, $K$ is not monogenic. For an other example, see \cite[Example 7.4.4]{Alacawiliamsbook}.
\end{remarks}

	  \begin{table}[h!]
	  	\centering
		\begin{tabular} { | c | c | c |  }
			\hline
			Case & Conditions & Factorization of $3\Z_K$  \\
			\hline
			C1 & $a \equiv 1\md{3} $ and $b \equiv 1 \md{3}$ & $[1, 4, 4]$  \\
			\hline
				C2 & $a \equiv 1\md{3} $ and $b \equiv -1 \md{3}$ & $[1, 4, 4]$   \\
			\hline
				C3 & $a \equiv -1\md{3} $ and $b \equiv 1 \md{3}$ & $[3, 6]$ \\
			
			\hline
			C4 & $a \equiv -1\md{3} $ and $b \equiv -1 \md{3}$ & $[3, 6]$ \\
			\hline
			C5 & $a \equiv 1\md{3} $ and $b \equiv 0 \md{3}$ & $[1, 4, 4]$ \\
			\hline
			C6 &  $a \equiv -1\md{3} $ and $b \equiv 0 \md{3}$ &$[1, 1, 1, 2, 2, 2 ]$  \\
			\hline
			C7 &   $9 \n_3(a)<8 \n_3(b)$ and $\n_3(a) \in \{1, 3,  5, 7\}$ & $[1,1^8]$  \\
			\hline
			C8 &  $9 \n_3(a)<8 \n_3(b),$ $\n_3(a) \in \{2, 6\}$ and $a_3 \equiv 1 \md{3}$ & $[1,2^4]$  \\
			\hline
			C9 &   $9 \n_3(a)<8 \n_3(b),$ $\n_3(a) \in \{2, 6\}$ and $a_3 \equiv -1 \md{3}$ & $[1, 1^4, 1^4]$  \\
			\hline
			C10 &  $\n_3(a)=4, \, \n_3(b) \ge 5$ and $b_3 \equiv 1 \md{3}$   & $[1, 2^2, 2^2]$  \\
			\hline
				C11 &  $\n_3(a)=4, \, \n_3(b) \ge 5$ and $b_3 \equiv -1 \md{3}$   & $[1, 1^4, 1^4]$  \\
					\hline
						C12 &  $9 \n_3(a)>8 \n_3(b)$ and $\n_3(b) \in \{1, 2, 4, 5, 7, 8\}$   & $[1^9]$  \\
					\hline
						C13 &   $9 \n_3(a)>8 \n_3(b)$ and $\n_3(b) \in \{3, 6\}$   & $[1^9], [1^3, 1^3, 1^3],$  \\
					& & $[1^3, 1^6]$, or $[1^3, 2^3]$ \\
					\hline
		C14 &  $(a, b) \in \{(0, 2), (0, 5), (3, -1), (3, 2), (6, -1), (6, 5)\}\md{9}$   & $[1^9]$   \\
	\hline	
		C15 &  $(a, b) \in \{(3, 5), (6, 2)\}\md{9}$   & $[1, 1^8]$  \\
	\hline	
		C16 &   $(a, b) \in \{(0, 8), (0, 18), (9, -1), $ & $[1^3, 1^6]$  \\
			 &   $ (9, 8), (18, -1), (18, 17)\}\md{27}$  &   \\
	\hline	
		C17 &  $(a, b) \in \{(0, -1), (9, 17)\}\md{27}$   & $[1, 1^2, 1^6]$  \\
	\hline	
		C18 &    $(a, b) \in \{(18, 8), (18, 35), (45, 8), (45, 62), (72, 35), (72, 35)\}\md{81}$ &  $[1^3, 1^6]$ \\
	\hline	
		C19 &  $(a, b) \in \{(18, 62), (99, 224), (180, 143)\}\md{243}$   & $[1, 2, 1^6]$  \\
	\hline				
			C20 & $(a, b) \in \{(45, 35), (126, 197), (207, 116)\}\md{243}$  & $[3, 1^6]$  \\
		\hline	
			C21 &  $(a, b) \in \{(18, 143), (99, 62), (180, 224)\}\md{243}$   & $[3, 1^6]$  \\
		\hline
			C22 &  $(a, b) \in \{(45, 116), (126, 35), (207, 197)\}\md{243}$   & $[1^3, 1^6]$  \\
		\hline
			C23 &  $(a, b) \in\{(45, 197), (126, 116),  (207, 35)\}\md{243}$&  $[1, 2, 1^6]$  \\ 
		\hline
			C24 &   $(a, b) \in \{(18, 224), (99, 143)\} \md{243}$ & $[3, 1^6]$\\
		\hline
			C25 &   $ a \equiv 18\md{243}$ and $b \equiv 62\md{243}$ &  $[1, 2, 1^6]$ \\
		\hline
			C26 &  $(a, b) \in \{(72, 8), (234, 89), (153, 170)\}\md{243}$&  $[1, 2, 1^6]$ \\ 
		\hline
			C27 &  $(a, b) \in \{(153, 8), (72, 89), (234, 170)\}\md{243}$&  $[3, 1^6]$ \\ 
		\hline
			C28 &  $(a, b) \in \{(72, 170), (234, 8)\}\md{243}$&  $[1, 1, 1, 1^6]$ \\ 
		\hline	
		C29 &  $ a \equiv 153 \md{143}, b \equiv 89 \md{243}, 2 \mu > \n +2$ and $\n$ is odd&  $[1^2, 1, 1^6]$ \\ 
		\hline	
			C30 &  $a \equiv 153 \md{143}, b \equiv 89 \md{243}$ and $ 2 \mu < \n +2$&  $[1, 1, 1,  1^6]$ \\ 
			\hline	
			C31 &  $a \equiv 153 \md{243}, b \equiv 89 \md{243},  2 \mu > \n +2,$ &  $[2, 1,  1^6]$ \\ 
				 &  $\n$ is even,  and $(1+a+b)_3\equiv 1 \md{3}$&   \\
			\hline
			C32 &  $a \equiv 153 \md{243}, b \equiv 89 \md{243},  2 \mu > \n +2,$ &  $[1, 1, 1,  1^6]$ \\
				 &  $\n$ is even,  and $(1+a+b)_3\equiv -1 \md{3}$&   \\  
			\hline
			C33 &  $a \equiv 153 \md{243}, b \equiv 89 \md{243},  2 \mu = \n +2$,  and &  $[2, 1,  1^6]$ \\ 
			 &    and $(1+a+b)_3\equiv -1 \md{3}$&   \\ 
				\hline
			C34 &  $a \equiv 153 \md{243}, b \equiv 89 \md{243},  2 \mu = \n +2,$ and   &  $[2, 1, 1^6]$ \\ 
			 &  $ (1+a+b)_3\equiv 1 \md{3},$   &   \\ 
				\hline
\end{tabular}
\end{table}

\begin{table}[h!]
	\centering
	\begin{tabular} { | c | c | c |  }
		\hline
		C35 &  $(a, b) \in \{(0, 4), (0, 7), (3, 1), (3, 7), (6, 1), (6, 4)\}\md{9}$&  $[1^9]$ \\ 
	\hline
	C36 &  $(a, b) \in \{(3, 4), (6, 7)\}\md{9}$&  $[1, 1^8]$ \\ 	
	\hline
	C37 &  $(a, b) \in \{(0, 10), (0, 19), (9, 1), (9, 19), (18, 1), (18, 10)\}\md{27}$&  $[1^3, 1^6]$ \\  
	\hline
	C38 &  $(a, b) \in \{(0, 1),  (9, 10)\}\md{27}$&  $[1, 1^2, 1^6]$ \\ 	
	\hline
	C39 &  $(a, b) \in \{(18, 46), (18, 73), (45, 19), (45, 73), (72, 19), (72, 46)\}\md{81}$&  $[1^3, 1^6]$ \\
	\hline
	C40 &  $(a, b) \in \{(18, 100), (99, 18), (18, 19), (72, 154)\} \md{243} $  & $[3, 1^6]$  \\
	\hline				
	C41 &  $(a, b) \in \{(45, 208), (126, 46), (207, 127), (153, 73), (153, 235)\} \md{243} $  & $[3, 1^6]$  \\
\hline
	C42 &  $a \equiv 234 \md{243} $ and $b \equiv 73 \md{243}$  & $[1, 1, 1, 1^6]$  \\
\hline	
C43 &  $a \equiv 234 \md{243} $ and $b \equiv 154 \md{243}$  & $[2, 1, 1^6]$  \\
\hline
C44 &  $(a, b) \in \{(45, 46), (126, 127), (207, 208), (153, 154)\} \md{243} $  & $[1, 2, 1^6]$  \\
\hline
C45 &  $a \equiv 234 \md{243}, b \equiv 235 \md{243} , 2 \mu > \omega +2$ and $\omega$ is odd&  $[1^2, 1, 1^6]$ \\ 
\hline	
C46 &  $a \equiv 234 \md{243}, b \equiv 235 \md{243}$ and $ 2 \mu < \omega +2$&  $[1, 1, 1,  1^6]$ \\ 
\hline	
C47 &  $a \equiv 234 \md{243}, b \equiv 235 \md{243},  2 \mu > \omega +2$, &  $[2, 1,  1^6]$ \\ 
&  $\omega$ is even,  and $(-1-a+b)_3\equiv -1 \md{3}$&   \\ 
\hline
C48 &  $a \equiv 234 \md{243}, b \equiv 235 \md{243},  2 \mu > \omega +2$,&  $[1, 1, 1,  1^6]$ \\ 
&  $\omega$ is even,  and $(-1-a+b)_3\equiv 1 \md{3}$&   \\ 
\hline
C49 &  $a \equiv 234 \md{243}, b \equiv 235 \md{243},  2 \mu = \omega +2$,  and &  $[2, 1,  1^6]$ \\ 
&   $(-1-a+b)_3\equiv 1 \md{3}$&   \\ 
\hline
	C50 &  $(a, b) \in \{(45, 127), (126, 208), (207, 46)\} \md{243} $  & $[3, 1^6]$  \\
\hline
C51 &  $(a, b) \in \{(99, 100), (180, 181), (72, 73)\} \md{243} $  & $[1, 2, 1^6]$  \\
\hline
	C52 &  $(a, b) \in \{(18, 181), (99, 19), (180, 100)\} \md{243} $  & $[1, 2, 1^6]$  \\
\hline
	C53 &  $a \equiv 72 \md{243} $ and $b \equiv 135 \md{243}$  & $[3,1^6]$  \\
\hline	
C54 &  $a \equiv 18 \md{243} $ and $b \equiv 19 \md{243}$  & $[1^3,  1^6]$  \\
\hline
	
\end{tabular}
\caption{Factorization of $3\Z_K$}
\label{table3}
\end{table} 
\newpage
\section{Preliminary results}
Let $K $ be a number field generated by $\th$, a root of a monic irreducible trinomial $F(x)= x^9+ax+b \in \Z[x]$ and $\Z_K$  its ring of integers. Let $p$ be a rational prime.  We start by stating   the  following Lemma  which gives a necessary and sufficient condition for a rational prime  $p$ to be a prime common index divisor of $K$. This Lemma will  play an important role in the proof of our results. It is a consequence of  Dedekind's theorem on factorization of primes in number fields (see \cite[Theorems 4.33 and 4.34 ]{Na} and \cite{R}). 
\begin{lemma} \label{comindex}
	Let  $p$ be a rational prime  and $K$  a number field. For every positive integer $f$, let $L_p(f)$ denote the number of distinct prime ideals of $\Z_K$ lying above $p$ with residue degree $f$ and $N_p(f)$ denote the number of monic irreducible polynomials of  $\F_p[x]$ of degree $f$. Then $p$ is a common index divisor of $K$ if and only if $L_p(f) > N_p(f)$ for some positive integer $f$.
\end{lemma}
	 
		 Recall  by \cite[Proposition 4.35]{Na} that  the number of  monic irreducible polynomials of degree $f$ in $\F_p[x]$ is
		$$N_p(f) = \frac{1}{f} \sum_{d \mid f} \mu (d) p^{\frac{f}{d}},$$
		where $\mu$ is the M\"{o}bius function. Note  also that $N_{p}(f+1)>N_{p}(f)$ (except one case when $p=2$ and $f=1$).  
Since the above  condition given in  Lemma \ref{comindex} for a  prime $p$ to divide $i(K)$  depends upon the factorization of $p$ in $\Z_K$, we will need to determine the number of distinct prime ideals of $\Z_K$ lying above $p$. We will use Newton polygon techniques. So, let us  shortly recall   some fundamental notions and results on this method. For more details, we refer to \cite{ EMN, Nar, Narprime, MN92, O}.  The reader can see also \cite{BRM, BATCA}.  
 Let $p$  be a rational prime  and $\n_p$  the discrete valuation of $\Q_p(x)$  defined on $\Z_p[x]$ by $$\n_p\left(\sum_{i=0}^{m} a_i x^i\right) = \min \{ \n_p(a_i), \, 0 \le i \le m\}.$$ Let $\phi(x) \in \mathbb{Z}[x]$ be a monic polynomial whose reduction modulo $p$ is irreducible. The polynomial $F(x) \in \mathbb{Z}[x]$ admits a unique $\phi$-adic development $$ F(x )= a_0(x)+ a_1(x) \phi(x) + \cdots + a_n(x) {\phi(x)}^n,$$ with $ \deg \ ( a_i (x))  < \deg \ ( \phi(x))$. For every $0\le i \le n,$   let $ u_i = \n_p(a_i(x))$. The $\phi$-Newton polygon of $F(x)$ with respect to $p$ is the  lower boundary convex envelope of the set  of points $ \{  ( i , u_i) \, , 0 \le i \le n \, , a_i(x) \neq 0  \}$ in the euclidean plane, which we denote   by $N_{\phi} (F)$. The polygon  $N_{\phi} (F)$ is the union of different adjacent sides $ S_1, S_2, \ldots , S_g$ with increasing slopes $ \lambda_1, \lambda_2, \ldots,\lambda_g$. We shall write $N_\phi(F) = S_1+S_2+\cdots+S_g$. The polygon determined by the sides of negative slopes of $N_{\phi}(F)$ is called the  $\phi$-principal Newton polygon of $F(x)$  with respect to $p$ and will be denoted by $\npp{F}$. The length of $\npp{F}$ is $ l(\npp{F}) = \nu_{\overline{\ph}}(\overline{F(x)})$;  the highest power of $\phi$ dividing $F(x)$ modulo $p$.\\
 Let $\mathbb{F}_{\phi}$ be the finite field   $ \mathbb{Z}[x]\textfractionsolidus(p,\phi (x)) \simeq \mathbb{F}_p[x]\textfractionsolidus (\overline{\ph(x)}) $.
 We attach to  any abscissa $ 0 \leq i \leq  l(\npp{F})$ the following residue coefficient $ c_i \in  \mathbb{F}_{\phi}$
 :

 $$c_{i}=
 \left
 \{\begin{array}{ll} 0,& \mbox{if }  (i , u_i ) \, \text{ lies  strictly  above }  \ \npp{F},\\
 \dfrac{a_i(x)}{p^{u_i}}
 \,\,
 \md{(p,\phi(x))},&\mbox{if } \  (i , u_i ) \, \text{lies on } \npp{F}. 
 \end{array}
 \right.$$Let $S$ be one of the sides of $\npp{F}$. Then the length of $S$, denoted $l(S)$ is the length of its  projection to the horizontal axis and its height, denoted $h(S)$ is the length of its  projection to the vertical axis.  Let   $\lambda = - \frac{h(S)}{l(S)}= - \frac{h}{e} $  its slope, where $e$ and $h$ are two positive coprime integers.  The degree of $S$ is $d(S) = \gcd(h(S), l(S))=  \frac{l(S)}{e}$; it is equal to the the number of segments into which the integral lattice divides $S$. More precisely, if $ (s , u_s)$ is the initial point of $S$, then the points with integer coordinates  lying in $S$ are exactly $$  (s , u_s) ,\ (s+e , u_s - h) , \ldots, (s+de , u_s - dh).$$  The natural integer $e= \frac{l(S)}{d(S)}$ is called the ramification index of the side $S$ and denoted by $e(S)$. Further, we attach to $S$ the following residual polynomial:  $$ R_{\l}(F)(y) = c_s + c_{s+e}y+ \cdots + c_{s+(d-1)e}y^{d- 1}+ c_{s+de}y^d \in \mathbb{F}_{\phi}[y].$$  
Now, we give some important definitions.
\begin{definitions}
	Let $F(x) \in \Z[x]$ be a monic irreducible polynomial.  Let $\overline{F(x)}=\prod_{i=1}^t\overline{\ph_i}(x)^{l_i}$ be  the factorization  of $\overline{F(x)}$ into a product  of powers of distinct monic irreducible polynomials in  $\mathbb{F}_p [x]$. For every $i=1,\dots,t$, let  $N_{\ph_i}^+(F)=S_{i1}+\dots+S_{ir_i}$, and for every {$j=1,\dots, r_i$},  let $R_{\l_{ij}}(F)(y)=\prod_{s=1}^{s_{ij}}\psi_{ijs}^{n_{ijs}}(y)$ be the factorization of $R_{\l_{ij}}(F)(y)$ in $\F_{\ph_i}[y]$. 
	\begin{enumerate}
	\item  For every $i=1,\dots,t$, the $\ph_i$-index of $F(x)$, denoted by $ind_{\ph_i}(F)$, is  deg$(\ph_i)$ multiplied by  the number of points with natural integer coordinates that lie below or on the polygon $N_{\ph_i}^{+}(F)$, strictly above the horizontal axis  and strictly beyond the vertical axis.
	\item The polynomial $F(x)$ is said to be $\phi_i$-regular with respect to $\n_p$ (or $p$) if for every $j=1,\dots, r_i$,  $R_{\l_{ij}}(F)(y)$ is separable; $n_{ijs}=1$.
	\item The polynomial $F(x)$ is said to be $p$-regular if it is $\phi_i$-regular for every $ 1 \leq i \leq t$.
	\end{enumerate}
\end{definitions} 
Now, we state Ore's theorem which will be often used in the proof of our theorems (see   \cite[Theorem 1.7 and Theorem 1.9]{EMN},  \cite{MN92} and \cite{O}):

 \begin{theorem}\label{ore} (Ore's Theorem)\
 	\\Let $K$ be a number field generated by  $\th$, a root of a monic irreducible polynomial $F(x) \in \Z[x]$.  Under the above notations, we have: 
 	\begin{enumerate}
 		\item
 		$$ \nu_p((\Z_K:\Z[\th]))\ge \sum_{i=1}^t ind_{\ph_i}(F).$$ Moreover, the  equality holds if $F(x)$ is $p$-regular
 		\item
 		If  $F(x)$ is $p$-regular, then 
 		$$p\Z_K=\prod_{i=1}^t\prod_{j=1}^{r_i}
 		\prod_{s=1}^{s_{ij}}\pF^{e_{ij}}_{ijs},$$ where $e_{ij}$ is the ramification index
 		of the side $S_{ij}$ and $f_{ijs}=\mbox{deg}(\ph_i)\times \mbox{deg}(\psi_{ijs})$ is the residue degree of $\mathfrak{p}_{ijs}$ over $p$.
 	\end{enumerate}
 \end{theorem}
The following result is an immediate consequence of the above theorem. 
 \begin{corollary}\label{corollaryore} Under the above hypotheses, then the following hold:
 \begin{enumerate}
 \item  If for some $i=1, \ldots, r$, $l_i=1$, then the factor $\ph_i(x)$ of $F(x)$ modulo $p$ provides a unique prime ideal of $\Z_K$ lying above $p$, of residue degree equals $\deg(\ph_i(x))$ and of  ramification index  equals $1$.
 \item If for some $i=1, \ldots, r$, $N_{\ph_i}^+(F)=S_{i1}$ has only one side of degree $1$, then the factor $\ph_i(x)$ of $F(x)$ modulo $p$ provides a unique prime ideal of $\Z_K$ lying above $p$, of residue degree equals $\deg(\ph_i(x))$ and  of ramification index equals   $l_i$.
 
 \end{enumerate}
 
 \end{corollary}
  As an example of application of   Ore's theorem, we propose the following one.
  
 \begin{example}\
 	\\Consider the  sextic number field $K=\Q(\th)$ with $\th$ a root of  $F(x) = x^6 + 24 x +15 \in \Z[x] $. Reducing modulo $2$, we see that   $ \ol{F(x)} =  \ol{(\ph_1(x) \cdot  \ph_2(x))}^2 $  in  $\F_2[x],$ where $\ph_1(x) = x-1$ and $\ph_2(x) = x^2+x+1$. The $\ph_1$-adic development of $F(x)$ is 
 	\begin{eqnarray*}\label{dev61phi}
 		F(x)= \ph_1(x)^6+6 \ph_1(x)^5+15 \ph_1(x)^4+20 \ph_1(x)^3 + 15 \ph_1(x)^2+30\ph_1(x)+40.
 	\end{eqnarray*}
 	Thus, $N_{\ph_1}^+(F)=S_{11}+S_{12}$ has two sides of degree $1$ each  joining the points $(0,3), (1,1)$ and $(2,0)$. 	Their attached residual polynomials  are
 	$R_{\l_{11}}(F)(y) = R_{\l_{12}}(F)(y)=   1 +y $ which are separable in $\F_{\ph_1}[y] \simeq \F_2[y] $ as they are of degree $1$. Thus, $F(x)$ is $\ph_1$-regular.
 	The $\ph_2$-adic development of $F(x)$ is\begin{eqnarray*} \label{dev61psi}
 		F(x)=\ph_2(x)^3-3x\ph_2(x)^2+(2x-2)\ph_2(x)+ 24x+16.
 	\end{eqnarray*}
 	Since	$ \n_2(24x+16)= \min(\n_2(24), \n_2(16)) \ge 3$. Thus,  $N_{\ph_2}^+(F) = S_{21}+S_{22}$ has two sides of degree $1$ each  (see FIGURE 1). Thus,  $R_{\l_{2k}}(F)(y)$ is irreducible over $\F_{ \ph_2}$ for $k = 1,2$.  So,  $F(x)$ is $\ph_2$-regular.  Hence, $F(x)$ is $2$-regular. By  Theorem \ref{ore},  one gets: $$\nu_2(ind(F))=\nu_2((\Z_K:\Z[\th])) = ind_{\ph_1}(F) + ind_{\ph_2}(F)=1+1=2 $$ and $$2\Z_K = \pF_{111} \cdot \pF_{121} \cdot \pF_{211} \cdot \pF_{221}, $$ where $\pF_{111}, \pF_{121}, \pF_{211}$ and $ \pF_{221}$ are four prime ideals of $\Z_K$ with  residue degrees $f(\pF_{111}/2)= f(\pF_{121}/2)=1 $ and $f(\pF_{211}/2)= f(\pF_{221}/2)=2$. Further, since there are  two prime ideals of $\Z_K$  of  residue degree $2$ each  lying above $2$ ($N_2(2)=2$), by using Lemma \ref{comindex}, the prime  $2$ is a  common index divisor of $K$. Consequently, $K$ is not  monogenic.
 \end{example}
 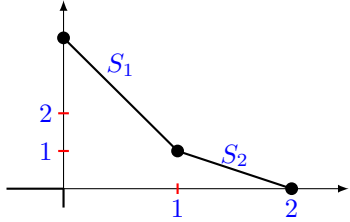
\begin{figure}[htbp]

 	\centering
 	
 	\begin{tikzpicture}[x=1.5cm,y=0.5cm]
 	\draw[latex-latex] (0,5) -- (0,0) -- (2.5,0) ;

 	\draw[thick] (0,0) -- (-0.5,0);
 	\draw[thick] (0,0) -- (0,-0.5);
 	
 	\draw[thick,red] (1,-2pt) -- (1,2pt);
 	\draw[thick,red] (2,-2pt) -- (2,2pt);
 	\draw[thick,red] (-2pt,1) -- (2pt,1);
 	\draw[thick,red] (-2pt,2) -- (2pt,2);
 	\node at (1,0) [below ,blue]{\footnotesize  $1$};
 	\node at (2,0) [below ,blue]{\footnotesize $2$};
 	\node at (0,1) [left ,blue]{\footnotesize  $1$};
 	\node at (0,2) [left ,blue]{\footnotesize  $2$};
 	\draw[thick, mark = *] plot coordinates{(0,4) (1,1) (2,0) };
 	\node at (0.5,2.7) [above  ,blue]{\footnotesize $S_{1}$};
 	\node at (1.5,0.3) [above   ,blue]{\footnotesize $S_{2}$};
 	\end{tikzpicture}
 	\caption{   $N_{\ph_2}^+(F)$ with respect to $\n_2$.}
 \end{figure}
 Since we will use Newton polygon techniques in second order to treat some cases,  we   briefly recall some  related  concepts of this algorithm in second order that we use throughout.\\
Let $p$ be a rational prime. Assume that  $F(x)$ is not regular with respect to $\n_p$.  Then we cannot apply Ore's Theorem. In their pioneer works,   Gu\'{a}rdia, Montes and 	Nart  introduced an efficient algorithm  to complete the factorization of   the principal ideal  $p\Z_K$ (see \cite{ Nar, Narprime}). They defined the Newton polygon of higher  orders, and  proved an extension of the theorems of the product,  the polygon,  the residual polynomial  and  of the index in any  order $r$. Let us look at what the news in the second order; $r=2$.  Let $\ph(x)$ be a monic irreducible factor of $F(x)$ modulo $p$ and  $S$   a side of $N_1 = N^{+}_{\ph}(F)$ with slope $\lambda = - \frac{h}{e}$, where $h$ and $e$ are two coprime positive integers. Assume that   $R_{\l}(F)(y)$ is not separable in $\F_{\ph}[y]$ and let   $\psi(y)$ be a  monic irreducible factor of $R_{\l}(F)(y)$ of degree $f$. A type of order $2$  is a data:						
$$ T= \bigl(\ph(x), \lambda, \psi(y);  \Phi_2(x)\bigr),$$ such that   $\Phi_2(x)$ is a monic irreducible polynomial in $\Z_p[x]$ of degree $m_2 = e\cdot f \cdot \deg(\ph(x))$ such that 
\begin{enumerate}
	\item  $N_1(\Phi_2)$ has only one side  of  slope $\lambda$.
	\item The residual polynomial  of $\Phi_2$ in first order,  $R_{\l}(\Phi_2)(y) \simeq  \psi_1(y)$ in  $\F_{\ph}[y]$ (up to multiply by a nonzero element of $\F_{\ph}$).
\end{enumerate}
If the above conditions hold, then $\Phi_2(x)$ induces a valuation  $\n^{(2)}_p$ on $\Q_p(x)$, called the augmented valuation of $\n_p$ of second order with respect to the type  $T$. Moreover, by  \cite[Proposition 2.7]{Nar}, if $P(x) \in \Z_p[x]$ such that  $P(x )= a_0(x)+ a_1(x) \ph(x) + \cdots + a_l(x) {\ph(x)}^l$, then $$\n^{(2)}_p(P(x)) = e \times  \min\limits_{0 \le j \le l}\{ \n_p(a_i(x)) + i  |\lambda|\}.$$  Let $F(x)= A_0(x)+ A_1(x) \Phi_2(x) + \cdots + A_t(x) {\Phi_2(x)}^t$ be the $\Phi_2$-adic development of $F(x)$ and  $\mu_i^{(2)} = \n^{(2)}_p(a_i(x) \Phi_2(x)^i) = \n^{(2)}_p(a_i(x))+ i \n^{(2)}_p(\Phi_2(x))  $, for every $0 \le i \le t$. The $\Phi_2$-Newton polygon of $F(x)$  of second order with respect $\n_p^{(2)}$  is the lower boundary of the convex envelope of the set of points   $ \{  ( i , \mu_i^{(2)}) \, , 0 \le i \le t  \}$ in the Euclidean plane, which we denote   by $N_{2} (F)$. Let $\F^{(2)}= \F_{\ph}[y]/\psi(y)$ be  the  residual field in second order associated to the type $T$.  As in first order, we define the corresponding residual polynomial for each side of  $N_{2} (F)$. We will use the theorem of the polygon and the  theorem of the residual polynomial in second order (\cite[Theorems 3.1 and 3.4]{Nar}). For more details,  we refer  to 
\cite{Nar} and \cite{Narprime} by   Gu\`{a}rdia,  Montes and  Nart.\\
 In order to treat some cases, we will use the following technical result. 
\begin{lemma}\label{technical}
Let $a$ be a rational integer such that $\n_2(a)=k=h \cdot 2^t$ for some  positive  integers $k, h$ and $t$. Let $K=\Q(\th),$ with $\th$ o root of a monic irreducible polynomial $F(x)=x^9+ax+b$. Then we can choose  the integer $a$, such that for every positive integer $i$, there exist a rational integer $\b_i \in \Z$ such that $ \n_2(\b_i) =h $ and $\n_2(\b_i^{2^t}+a)\ge k+ i$.
\end{lemma}
Now, we prove our first main theorem.
\begin{proof} [Proof of Theorem \ref{p=2}]\
	\\ \textbf {A1:} $a \equiv 1 \md{2}$ and $b \equiv 1 \md{2}$. In this case,  $\ol{F(x)}$ is irreducible    in  $\F_2[x]$. By Corollary \ref{ore}(1),    $2 \Z_K$ is a prime ideal of $\Z_K$ of residue degree $9$. By Lemma \ref{comindex}, $2$ does not divide $i(K)$.\\
	\textbf {A2:} $a \equiv 0 \md{2}$ and $b \equiv 1 \md{2}$. In this case,  $\ol{F(x)}= \ol{\ph_1(x)\ph_2(x)\ph_3(x)}$  in  $\F_2[x]$, where $\ph_1(x)=x+1, \ph_2(x)=x^2+x+1$ and $\ph_3(x)= x^6+x^3+1$.  By Corollary \ref{ore}(1), we have    $2 \Z_K= \pF_{111}\cdot \pF_{121}\cdot \pF_{131}$ with respective residue degrees  $f(\pF_{111}/2)=1, f(\pF_{121}/2)=2$ and $f(\pF_{131}/2)=6$.
	
\noindent $\bullet$ In  Cases  \textbf{A3-A10}, we have  $2$ divides $b$ and does not divide $a$.  Reducing modulo $2$, we see that $\ol{F(x)}= \ol{\ph_1(x) \ph_2(x)^8}$ in $\F_2[x]$, where $\ph_1(x)=x$ and $\ph_2(x)=x-1$. As $\n_{\ol{\ph_1(x)}}(\ol{F(x)})=1$, by Corollary \ref{corollaryore}(1),  the factor $\ph_1(x)$ of $F(x)$ modulo $2$ provides one prime ideal lying above $2$ of residue degree $1$ and of ramification index $1$, say $\pF_{111}$. It follows that $2\Z_K= \pF_{111}\cdot \aF,$ where $\aF$ is a non-zero ideal of $\Z_K$ provided by the factor $\ph_2(x)$. The $\ph_2$-adic development of $F(x)$ is 
	\begin{eqnarray}\label{devfi2}
	F(x)&=&\ph_2(x)^9+9\ph_2(x)^8+36\ph_2(x)^7+84\ph_2(x)^6+126\ph_2(x)^5+126\ph_2(x)^4+84\ph_2(x)^3 \nonumber \\
	&+&36\ph_2(x)^2+(9+a)\ph_2(x)+1+a+b.
	\end{eqnarray}
Let $\mu_0= \n_2(1+a+b)$	and $\mu_1=\n_2(9+a)$. It follows that  by the above $\ph_2$-adic development of $F(x)$, the $\ph_2$-principal Newton polygon of $F(x)$ is the lower convex hull of the points $(0, \mu_0), (1, \mu_1), (2, 2), (3, 2), (4, 1), (5, 1), (6, 2), (7, 2)$ and $(8, 0)$.\\ 
\textbf{A3:} $(a,b) \in \{(1, 0), (3, 2)\} \md{4}$. In this case, $\mu_0=1$. It follows by $(\ref{devfi2})$ that $N_{\ph_2}^{+}(F)=S_{21}$ has only one side of degree $1$ joining the points $(0, 1)$ and $(8, 0)$ with ramification index $e_{21}=8$. By Corollary \ref{corollaryore}(2), we see that  $2\Z_K = \pF_{111} \cdot \pF_{211}^8$ with $f(\pF_{111}/2)=f(\pF_{211}/2)= 1$. Therefore, $2$ does not divide $i(K)$. \\
\textbf{A4:}  $a \equiv 1 \md{4}$ and $b \equiv 2 \md{4}$. In this case,  $\mu_1=1$ and $\mu_0 \ge 2$. It follows by $(\ref{devfi2})$ that $$N_{\ph_2}^{+}(F)=S_{21}+S_{22}$$ has two sides of degree $1$ each with respective slopes $\l_{21} \le -1$ and $\l_{22}=\frac{-1}{7}$. More Precisely,  $N_{\ph_2}^{+}$ is the lower convex hull of the points $(0, \mu_0), (1, 1)$ and $(8, 0)$ (see FIGURE 2). In this case, the residual polynomials $R_{\l_{21}}(F)(y)$ and $R_{\l_{22}}(F)(y)$ are irreducible over $\F_{ \ph_2} \simeq \F_2$ as they are of degree $1$ each. So, $F(x)$ is $\ph_2(x)$-regular. Hence it is $2$-regular. Applying Theorem  \ref{ore}, we see that 
$$2\Z_K = \pF_{111} \cdot \pF_{211} \cdot \pF_{221}^7,$$ where $\pF_{111}$,  $\pF_{211}$ and  $\pF_{221}$ are three prime ideals of $\Z_K$ of residue degree $1$ each. It follows that $L_2(1)=3$ (there are three prime ideals of $\Z_K$ of residue degree $1$ each lying above the prime $2$). As $N_2(1)=2$ (there are only two monic  irreducible polynomials of degree $1$ in $\F_2[x]$, namely $x$ and $x+1$), $L_2(1)=3 > 2= N_2(1)$. By Lemma \ref{comindex}, $2$ divides $i(K)$.	\begin{figure}[htbp] 
	\centering	
	\begin{tikzpicture}[x=0.75cm,y=1cm]
	\draw[latex-latex] (0,2.5) -- (0,0) -- (8.5,0) ;

	\draw[thick] (0,0) -- (-0.5,0);
	\draw[thick] (0,0) -- (0,-0.5);
	
	\draw[thick,red] (1,-2pt) -- (1,2pt);
	\draw[thick,red] (2,-2pt) -- (2,2pt);
	\draw[thick,red] (3,-2pt) -- (3,2pt);
	\draw[thick,red] (4,-2pt) -- (4,2pt);
	\draw[thick,red] (5,-2pt) -- (5,2pt);
	\draw[thick,red] (6,-2pt) -- (6,2pt);
	\draw[thick,red] (7,-2pt) -- (7,2pt);
	\draw[thick,red] (8,-2pt) -- (8,2pt);
	\draw[thick,red] (-2pt,1) -- (2pt,1);
	\draw[thick,red] (-2pt,2) -- (2pt,2);
	\node at (1,0) [below ,blue]{\footnotesize  $1$};
	\node at (2,0) [below ,blue]{\footnotesize $2$};
	\node at (3,0) [below ,blue]{\footnotesize  $3$};
	\node at (4,0) [below ,blue]{\footnotesize  $4$};
	\node at (5,0) [below ,blue]{\footnotesize  $5$};
	\node at (6,0) [below ,blue]{\footnotesize  $6$};
	\node at (7,0) [below ,blue]{\footnotesize $7$};
	\node at (8,0) [below ,blue]{\footnotesize  $8$};
	\node at (0,1) [left ,blue]{\footnotesize  $1$};
	\node at (0,2) [left ,blue]{\footnotesize  $\mu_0 \ge 2$};
	\draw[thick, mark = *] plot coordinates{(0,2) (1,1) (8,0)};
	\draw[thick, only marks, mark=*] plot coordinates{ (2,2) (3,2) (4,1) (5,1) (6,2) (7,2)};	
	\node at (0.5,1.5) [above  ,blue]{\footnotesize $S_{21}$};
	\node at (3,0.6) [above   ,blue]{\footnotesize $S_{22}$};
	\end{tikzpicture}
	\caption{   $N_{\ph_2}^{+}$ with respect to $\n_2$ when  $a\equiv 1 \md 4$ and $b \equiv 2 \md 4$.}
\end{figure}
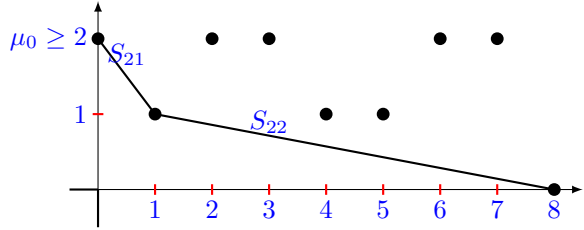\\
\textbf{A5:} $(a,b) \in \{(3, 0), (7, 4)\} \md{8}$.  In this case, we have $\mu_0=2$ and $\mu_1 \ge 2$.   Thus, $N_{\ph_2}^{+}(F)=S_{21}$ has only one of degree $2$ joining $(0, 2), (4, 1)$ and $(8, 0)$ with ramification index $e_{21}=4$.  Further, we have $R_{\l_{21}}(F)(y)=y^2+y+1$ which is irreducible in $\F_{ \ph_2} [y] \simeq \F_2[y]$. Thus, $F(x)$ is $2$-regular. By Using Theorem \ref{ore}, we obtain that $2\Z_K = \pF_{111} \cdot \pF_{211}^4,$ with $f(\pF_{111})=1$ and $f(\pF_{211} )=2$. So, $2$ is not a prime common index divisor of $K$.   \\
\textbf{A6:} $a \equiv 3 \md{8}$ and $b \equiv 4 \md{8}$. In this case, we have  $\mu_0 \ge 3$, $\mu_1=2$. Thus, $$N_{\ph_2}^{+}(F)=S_{21}+S_{22}+S_{23}$$ has three sides of degree $1$ each with respective slopes $\l_{21} \le -1$, $\l_{22}=\frac{-1}{3}$ and $\l_{23}=\frac{-1}{4}$. More explicitly,  $N_{\ph_2}^{+}$ is the lower convex hull of the points $(0, \mu_0), (1, 2)$ and $(8, 0)$ (see FIGURE 3).
The residual polynomials $R_{\l_{2k}}(F)(y), \, i=1,2,3$ are separable as they are of degree $1$ each. Thus, $F(x)$ is $2$-regular. So, Theorem \ref{ore} is applicable. Using this theorem, we have $$2\Z_K =  \pF_{111}\cdot  \pF_{211} \cdot  \pF_{221}^3 \cdot  \pF_{231}^4, $$ with $$f(\pF_{111}/2)=f(\pF_{211}/2)=f(\pF_{221}/2)=f(\pF_{231}/2) =1.$$ Thus, there are four prime ideals of $\Z_K$ of residue degree $1$ each lying above $2$ ($L_2(1)=4$). By Lemma \ref{comindex}, $2$ divides $i(K)$.	\begin{figure}[htbp] 
	\centering	
	\begin{tikzpicture}[x=1cm,y=0.75cm]
	\draw[latex-latex] (0,3.5) -- (0,0) -- (8.5,0) ;

	\draw[thick] (0,0) -- (-0.5,0);
	\draw[thick] (0,0) -- (0,-0.5);
	
	\draw[thick,red] (1,-2pt) -- (1,2pt);
	\draw[thick,red] (2,-2pt) -- (2,2pt);
	\draw[thick,red] (3,-2pt) -- (3,2pt);
	\draw[thick,red] (4,-2pt) -- (4,2pt);
	\draw[thick,red] (5,-2pt) -- (5,2pt);
	\draw[thick,red] (6,-2pt) -- (6,2pt);
	\draw[thick,red] (7,-2pt) -- (7,2pt);
	\draw[thick,red] (8,-2pt) -- (8,2pt);
	\draw[thick,red] (-2pt,1) -- (2pt,1);
	\draw[thick,red] (-2pt,2) -- (2pt,2);
	\node at (1,0) [below ,blue]{\footnotesize  $1$};
	\node at (2,0) [below ,blue]{\footnotesize $2$};
	\node at (3,0) [below ,blue]{\footnotesize  $3$};
	\node at (4,0) [below ,blue]{\footnotesize  $4$};
	\node at (5,0) [below ,blue]{\footnotesize  $5$};
	\node at (6,0) [below ,blue]{\footnotesize  $6$};
	\node at (7,0) [below ,blue]{\footnotesize $7$};
	\node at (8,0) [below ,blue]{\footnotesize  $8$};
	\node at (0,1) [left ,blue]{\footnotesize  $1$};
	\node at (0,3) [left ,blue]{\footnotesize  $\mu_0 \ge 3$};
	\node at (0,2) [left ,blue]{\footnotesize  $2$};
	\draw[thick, mark = *] plot coordinates{(0,3) (1,2) (4,1) (8,0)};
	\draw[thick, only marks, mark=*] plot coordinates{ (2,2) (3,2) (4,1) (5,1) (6,2) (7,2)};	
	\node at (0.5,2.5) [above  ,blue]{\footnotesize $S_{21}$};
	\node at (3,1.2) [above   ,blue]{\footnotesize $S_{22}$};
	\node at (6,0.4) [above   ,blue]{\footnotesize $S_{23}$};
	\end{tikzpicture}
	\caption{   $N_{\ph_2}^{+}$ with respect to $\n_2$ when  $a\equiv 3 \md 8$ and $b \equiv 4 \md 8$.}
\end{figure}
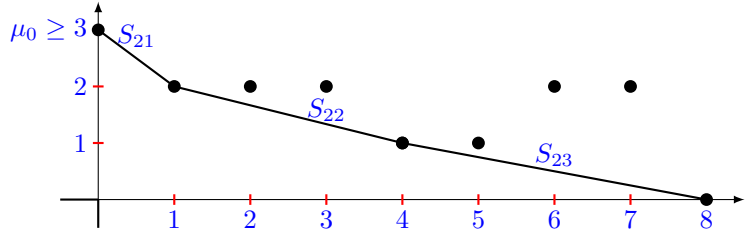\\
\textbf{A7:} $(a,b) \in \{(7, 0), (15, 8)\} \md{16}$. In this case, $\mu_0=3$ and $\mu_1 \ge 3$. According to the $\ph_2$-adic development $(\ref{devfi2})$ of $F(x)$, we get  $N_{\ph_2}^{+}(F)=S_{21}+S_{22}$ has two sides such that $d(S_{21})=2, e_{22}=2, d(S_{22})=1, e_{21}=4$, and $ R_{\l_{21}}(F)(y)  =y^2+y+1, R_{\l_{22}}(F)(y)=1+y \in \F_{ \ph_2}[y]$. Thus, $F(x)$ is $2$-regular. By applying Theorem \ref{ore}, we see that $2\Z_K = \pF_{111} \cdot \pF_{211}^2\cdot \pF_{221}^4,$ with ramification indices $f(\pF_{111}/2)=f(\pF_{221}/2)=1$ and $ f(\pF_{211}/2)=2$.  By Lemma \ref{comindex}, $2$
 does not divide $i(K)$.\\
\textbf{A8:}   $(a,b) \in \{(15, 0), (31, 16)\} \md{32}$. In this case $\mu_0=4$ and $\mu_1=3$. By $(\ref{devfi2})$, we have $$N_{\ph_2}^{+}(F)=S_{21}+S_{22}+S_{23}$$ has three sides joining the points 
$(0, 4), (1, 3), (2, 2), (4, 1)$ and $(8, 0)$ with respective degrees $d(S_{21})=2$ and $d(S_{22})=d(S_{23})=1$ (see FIGURE 4). Further, the residual polynomials are $R_{\l_{21}}(F)(y)=1+y+y^2, \,R_{\l_{22}}(F)(y)=R_{\l_{23}}(F)(y)=1+y $ which are separable  in $\F_{ \ph_2} [y] \simeq \F_2[y]$. Applying Theorem \ref{ore}, we get $$2\Z_K= \pF_{111} \cdot  \pF_{211}\cdot \pF_{221}^2 \cdot\pF_{231}^4, $$ with residue degrees $$f(\pF_{111}/2)=f(\pF_{221}/2)=f(\pF_{231}/2) =1 \, \mbox{and}\,\, f(\pF_{211}/2)=2.$$  It follows that  $L_2(1)=3 > N_2(1)=2$. By Lemma \ref{comindex}, $2$ divides $i(K)$.\\
\textbf{A9:} $(a,b) \in \{(15, 16), (31, 0)\} \md{32}$. In this case,  $\mu_0 \ge 5$ and $\mu_1=3$. Thus, $$N_{\ph_2}^{+}(F)=S_{21}+S_{22}+S_{23}+S_{24}$$ has four  sides of degree $1$ each joining the points $(0, \mu_0), (1, 3), (2, 2), (4, 1)$ and $(8, 0)$ (see FIGURE 4). Moreover, the residual polynomials  $R_{\l_{2k}}(F)(y), \, k=1,2,3,4$  are irreducible  $\F_{ \ph_2} [y]$ as they are of degree $1$ each. So, Theorem \ref{ore} is applicable. Therefore, we have $$2\Z_K =  \pF_{111} \cdot  \pF_{211} \cdot  \pF_{221} \cdot \pF_{231}^2 \cdot \pF_{241}^4, $$ with $$f(\pF_{111}/2)=f(\pF_{211}/2)=f(\pF_{221}/2)=f(\pF_{231}/2) = f(\pF_{241}/2)=1.$$ 
So, $L_2(1)=5 > N_2(1)=2$. Hence, by Lemma \ref{comindex}, $2$ divides $i(K)$.\\
	\begin{figure}[htbp] 
	\centering
	\begin{tikzpicture}[x=0.8cm,y=0.5cm]
	\draw[thick] (-0.5,0) -- (9,0);
	\draw[thick] (0,-0.2) -- (0,5);
	\draw (0,0) node {-};
	\draw (0,0) node[left]{$0$};
	\draw (0,1) node {-};
	\draw (0,1) node[left]{$1$};
	\draw (0,2) node {-};
	\draw (0,2) node[left]{$2$};
	\draw (0,3) node {-};
	\draw (0,3) node[left]{$3$};
	\draw (0,4) node {-};
	\draw (0,4) node[left]{$4$};
	\draw (1,0) node {$\shortmid$};
	\draw (1,0) node[below] {$1$};
	\draw (2,0) node {$\shortmid$};
	\draw (2,0) node[below] {$2$};
	\draw (3,0) node {$\shortmid$};
	\draw (3,0) node[below] {$3$};
	\draw (2,2) node {$\bullet$};
	\draw (4,0) node[below] {$4$};
	\draw (4,0) node {$\shortmid$};
	\draw (5,0) node[below] {$5$};
	\draw (5,0) node {$\shortmid$};
	\draw (6,0) node[below] {$6$};
	\draw (6,0) node {$\shortmid$};
	\draw (7,0) node[below] {$7$};
	\draw (7,0) node {$\shortmid$};
	\draw (8,0) node[below] {$8$};
	\draw (8,0) node {$\shortmid$};
	\draw (4,1) node {$\bullet$};
	\draw(8,0)--(4,1);
	\draw(4,1)--(2,2);
	\draw(0,4)--(2,2);
	\draw (1,3)node{$\bullet$};
	\draw (0,4)node{$\bullet$};
	\draw (4,-2)node{$N_{\ph_1}^{+}(F)$,  $(a,b) \in \{(15, 0), (31, 16)\} \md{32}$};
	\draw[thick] (10,0) -- (19,0);
	\draw[thick] (10,-0.2) -- (10,6);
	\draw (10,1) node {-};
	\draw (10,1) node[left]{$1$};
	\draw (10,2) node {-};
	\draw (10,2) node[left]{$2$};
	\draw (10,3) node {-};
	\draw (10,3) node[left]{$3$};
	\draw (10,5) node {-};
	\draw (10,5) node[left]{$\mu_0$};
	\draw (10,0) node {$\shortmid$};
	\draw (10,0) node[below] {$0$};
	\draw (11,0) node {$\shortmid$};
	\draw (11,0) node[below] {$1$};
	\draw (12,0) node {$\shortmid$};
	\draw (12,0) node[below] {$2$};
	\draw (13,0) node {$\shortmid$};
	\draw (13,0) node[below] {$3$};
	\draw (14,0) node {$\shortmid$};
	\draw (14,0) node[below] {$4$};
	\draw (15,0) node {$\shortmid$};
	\draw (15,0) node[below] {$5$};	
	\draw (16,0) node {$\shortmid$};
	\draw (16,0) node[below] {$6$};
	\draw (17,0) node {$\shortmid$};
	\draw (17,0) node[below] {$7$};	
	\draw (18,0) node {$\shortmid$};
	\draw (18,0) node[below] {$8$};
	\draw (18,0) node {$\bullet$};
	\draw (12,2) node {$\bullet$};
	\draw (14,1)node {$\bullet$};
	\draw (10,5)node {$\bullet$};
	\draw (11,3)node {$\bullet$};
	\draw(14,1)--(18,0);
	\draw (12,2)--(14,1);
	\draw (10,5)--(11,3);
	\draw (11,3)--(12,2);
	\draw (14,-2)node{\small $N_{\ph_2}^{+}(F)$, $(a,b) \in \{(15, 16), (31, 0)\} \md{32}$};
%
	\end{tikzpicture}
		\caption{\large  $N_{\ph_2}^+(F)$ in Cases  \textbf{A8} and \textbf{A9}}
\end{figure}
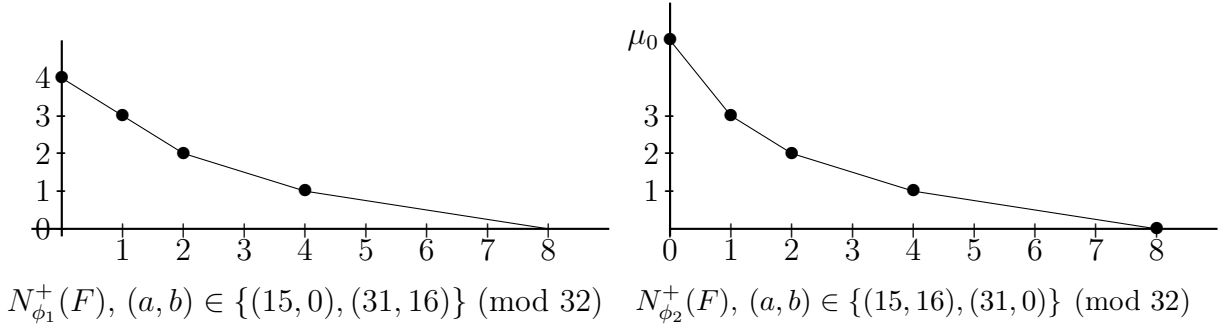

\textbf{A10:} $a \equiv 7 \md{16}$ and $b \equiv 8 \md{16}$. In this case, $\mu_0 \ge 4$ and $\mu_1 \ge 4$. Then, $N_{\ph_2}^{+}(F)=S_{21}+ \cdots+ S_{2, \, t-1}+ S_{2, t}$ has $t$ sides with $t \in \{3,4\}$. The last two sides have degree $1$ each joining the points $(2,2), (4,1)$ and $(8,0)$ with ramification indices equal  $ e_{1, \, t-1}=2$ and $ e_{1, t}=4$  (see FIGURE 4). Applying 	Theorem \ref{ore} and using the Fundamental Equality,  we see that $$2\Z_K =  \pF_{111} \cdot \pF_{2, t-1}^2 \cdot \pF_{2, t}^4 \cdot  \aF, $$ where $ \pF_{111}$,  $\pF_{2, t-1}$ and  $\pF_{2, t}$ are three prime ideals of residue degree $1$ each and $\aF$ is a non-zero ideal of $\Z_K$ provided by the other side(s)  of $N_{\ph_2}^{+}(F)$. Thus, $L_3(1) \ge 3 > 2=N_2(1)$. By Lemma \ref{comindex}, $2$ divides $i(K)$. 
 	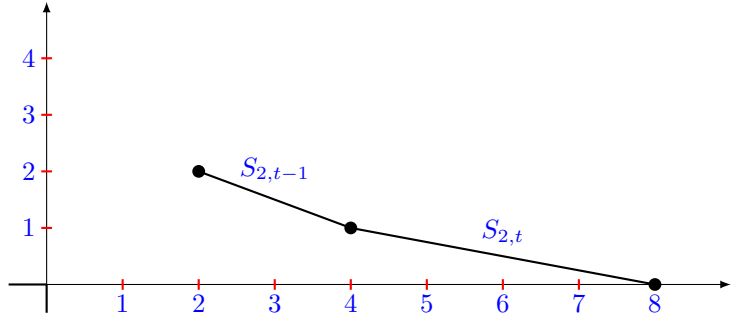
\begin{figure}[htbp] 
 	\centering
 	\begin{tikzpicture}[x=1cm,y=0.75cm]
 	\draw[latex-latex] (0,5) -- (0,0) -- (9,0) ;

 	\draw[thick] (0,0) -- (-0.5,0);
 	\draw[thick] (0,0) -- (0,-0.5);
 	
 	\draw[thick,red] (1,-2pt) -- (1,2pt);
 	\draw[thick,red] (2,-2pt) -- (2,2pt);
 	\draw[thick,red] (3,-2pt) -- (3,2pt);
 	\draw[thick,red] (4,-2pt) -- (4,2pt);
 	\draw[thick,red] (5,-2pt) -- (5,2pt);
 	\draw[thick,red] (6,-2pt) -- (6,2pt);
 	\draw[thick,red] (7,-2pt) -- (7,2pt);
 	\draw[thick,red] (8,-2pt) -- (8,2pt);
 	\draw[thick,red] (-2pt,1) -- (2pt,1);
 	\draw[thick,red] (-2pt,2) -- (2pt,2);
 	\draw[thick,red] (-2pt,3) -- (2pt,3);
 	\draw[thick,red] (-2pt,4) -- (2pt,4);	
 	\node at (1,0) [below ,blue]{\footnotesize  $1$};
 	\node at (2,0) [below ,blue]{\footnotesize $2$};
 	\node at (3,0) [below ,blue]{\footnotesize  $3$};
 	\node at (4,0) [below ,blue]{\footnotesize  $4$};
 	\node at (5,0) [below ,blue]{\footnotesize  $5$};
 	\node at (6,0) [below ,blue]{\footnotesize  $6$};
 	\node at (7,0) [below ,blue]{\footnotesize $7$};
 	\node at (8,0) [below ,blue]{\footnotesize  $8$};
 	\node at (0,1) [left ,blue]{\footnotesize  $1$};
 	\node at (0,2) [left ,blue]{\footnotesize  $2$};
 	\node at (0,3) [left ,blue]{\footnotesize  $3$};
 	\node at (0,4) [left ,blue]{\footnotesize  $4$};
 	\draw[thick, mark = *] plot coordinates{ (2,2) (4,1) (8,0)};
 	\draw[thick, only marks, mark=*] plot coordinates{  };	
 	\node at (3,1.6) [above  ,blue]{\footnotesize $S_{2,t-1}$};
 	\node at (6,0.5) [above  ,blue]{\footnotesize $S_{2, t}$};
 	\end{tikzpicture}
 	\caption{  $N_{\ph_2}^{+}(F)$,  when $a \equiv 7 \md{16}$ and $b \equiv 8 \md{16}$}
 \end{figure}	\\
$\bullet$ From Case \textbf{A11}, we have  $2$ divides both  $a$ and $b$, $\ol{F(x)}=\ol{\ph_1(x)}^9$ in $\F_2[x]$, where $\ph_1(x)=x$.  It follows that the $\ph_1$-principal Newton polygon with respect $\n_2$; $N_{\ph_1}^{+}(F)$,  is the lower convex hull of the points $(0, \n_2(b)), (1, \n_2(b))$ and $(9, 0)$. Notice that the  finite residual field $\F_{\ph_1} \simeq \F_2$.  Recall also  that  by Hypothesis $(\ref{Hypothese})$, we have  $\n_2(a)<8$ or $\n_2(b)<9$. \\
\textbf{A11:}  $\n_2(b) \in\{1, 2, 4, 5, 7, 8\}$ and  $9
\n_2(a)>8\n_2(b)$. In this case, $N_{\ph_1}^{+}(F)=S_{11}$ has only one side of degree $1$,  joining the points $(0, \n_2(b))$ and $(9, 0)$. Its ramification index equals $9$.  By Corollary \ref{corollaryore}(2), we see that $2\Z_K=\pF_{111}^9,$ where $\pF_{111}$ is a prime ideal of $\Z_K$ of residue degree $1$. Therefore, by Lemma \ref{comindex}, $2$ does not divide $i(K)$. \\
\textbf{A12:} $\n_2(b) \in \{3, 6\}$ and  $9 \n_2(a)>8\n_2(b)$. Here, $N_{\ph_1}^{+}(F)=S_{11}$ has only one side of degree $3$,  joining the points $(0, \n_2(b))$ and $(9, 0)$. Its ramification index equals $3$. Further, $R_{\l_{11}}(F)(y)=y^3-1=(y-1)(y^2+y+1) \in \F_{\ph_1}[y]$. Thus, $F(x)$ is $2$-regular. Applying Theorem \ref{ore}, we get $3\Z_K= \pF_{111}^3\cdot \pF_{112}^3$, where $f(\pF_{111}/2)=1$ and $f(\pF_{112}/2)=2$. Hence, by Lemma \ref{comindex}, $2$ does not divide $i(K)$. \\
$\bullet$ From Case \textbf{A13},  we have  $9 \n_2(a)<8\n_2(b)$. Thus,  $N_{\ph_1}^{+}(F)=S_{11}+S_{12}$ has two sides joining the points $(0, \n_2(b)), (1, \n_2(b))$ and $(9, 0)$. The first side; $S_{11}$,  has degree $1$. So, $R_{\l_{11}}(F)(y)$ is separable. It follows that $2\Z_K = \pF_{111}\cdot\aF,$ where $\pF_{111}$ is a prime ideal     of $\Z_K$ of residue degree $1$, and $\aF$ is a non-zero ideal of $\Z_K$ provided by  the second side of $N_{\ph_1}^{+}(F)$; $S_{12}$.    In what follows, we determine the number of  prime  ideals of $\Z_K$ lying above $2\Z_K$, provided by the side   $S_{12}$.         \\
\textbf{A13:} $\n_2(a) \in\{1, 3,  5, 7\}$ and  $9 \n_2(a)<8\n_2(b)$. In this case $d(S_{12})=1$ and $e_{12}=8$. By corollary \ref{corollaryore}(2), we obtain that $2\Z_K= \pF_{111}\cdot \pF_{121}^8,$ where $\pF_{111}$ and $\pF_{121}$ are two prime ideals     of $\Z_K$ of residue degree $1$ each. Consequently, by Lemma \ref{comindex}, $2$ is not a common index divisor of $K$. \\
\textbf{A14:}  $\n_2(a)=2$ and  $\n_2(b) \in \{3, 4\}$.  Here, $d(S_{12})=2, e_{12}=4$ and $\l_{12}=\frac{-1}{4}$. But, $R_{\l_{12}}(F)(y)=(\psi(y))^2, $ where $\psi(y)=y-1$. Thus, $R_{\l_{12}}(F)(y)$ is not separable in $\F_{\ph_1}[y]$. In this cases Ore's Theorem (Theorem \ref{ore}) is not applicable.  But, we have the partial following information: $2\Z_K=\pF_{111}\cdot \aF^4, $  where $\pF_{111}$ is a prime ideal     of $\Z_K$ of residue degree $1$, and $\aF$ is a non-zero ideal of $\Z_K$. To factorize the ideal $\aF$, we use Newton polygon in second order.    Let   $\Phi_2(x)=x^4-\b,$ with $\b \in \Z$ such that $\n_2(\b)=2$. By Lemma \ref{technical}, we choose this $\b$ such that $\n_2(\b^2+a) \ge 5$. Consider the data $$T=(\ph_1(x), \l_{12}, \psi(y); \Phi_2(x)).$$ One can check that $T$ is a type of order 2. Let $\n_2^{(2)}$ be the valuation of second order with respect to $T$.  The $\Phi_2$-adic development of $F(x)$ is 
\begin{equation}\label{Deva=2}
F(x)=x\Phi_2(x)^2+2 \b x\Phi_2(x)+(\b^2+a)+b. 
\end{equation}
Let $A_0(x)=(\b^2+a)+b.$
By using  \cite[Theorem 2.11 and Proposition 2.7]{Nar}, we have
\begin{itemize}
	\item $\n_2^{(2)}(x)= 1,$
	\item $\n_2^{(2)}(\Phi_2)=4,$
	\item $\n_2^{(2)}(x\Phi_2(x)^2)=9,$
	\item $ \n_2^{(2)}(2 \b x\Phi_2(x))=13,$
	\item $ \n_2^{(2)}(A_0(x))=\min \{4\n_2(\b^2+a)+1, 4\n_2(b)\} \ge \min\{21, 4 \n_2(b)\}.$
\end{itemize}
Thus, $N_2(F)$;  the $\Phi_2$-Newton polygon  of second order with respect $\n_2^{(2)}$, is the lower convex hull of the points $(0, \mu_0^{(2)}), (1, 13)$ and $(2, 9)$, where $\mu_0^{(2)}= \n_2^{(2)}(A_0(x))$. Note also that  the residual field $\F^{(2)}= \F_{\ph_1}[y]/\psi(y)$  isomorphic to $\F_2$.\\
If $\n_2(b)=3$. Thus, $\mu_0^{(2)}=12$.   It follows that $N_2(F)=S_{11}^{(2)}$ has only one side of degree $1$ joining the points $(0, 12)$ and $(2, 9)$ with slope $\l_{11}^2=\frac{-1}{2}$. Further, we have $R_{\l_{11}^2}^{(2)}(F)(y)=z+1 \in \F^{(2)}[z]$. Thus,  $F(x)$ is regular with respect $\n_2^{(2)}$.  Using Theorems 3.1 and 3.4 of \cite{Nar}, we see that $2\Z_K=\pF_{111} \cdot \pF_{121}^8,$  where  $\pF_{111}$ and $\pF_{121}$ are two prime ideals     of $\Z_K$ of residue degree $1$ each. Similarly, if $\n_2(b)=4$, then $\mu_0^{(2)}=16$. Also, we see that the form of factorization of $2\Z_K$ is $[1, 1^8]$.  Consequently, by Lemma \ref{comindex}, $2$ does divides $i(K)$. \\
\textbf{A15:} $\n_2(a)=2$ and  $\n_2(b)\ge 5$. Here, the polygon $N_{\ph_1}^{+}(F)$ is the same as in the above case. Let $\Phi_2(x)$ be as in Case \textbf{A14}. In this case, we have $\mu_0^{(2)} \ge 20$. It follows by $(\ref{Deva=2})$ that   $N_2(F)=S_{11}^{(2)}+S_{11}^{(2)} $ has two sides of degree $1$ each. By applying  Theorems 3.1 and 3.4 of \cite{Nar}, we obtain that $2\Z_K=\pF_{111}\cdot \pF_{121}^4\cdot\pF_{122}^4,$ where $\pF_{111}, \pF_{121}$ and $\pF_{122}$ are three prime ideals     of $\Z_K$ of residue degree $1$ each.   So, we have $L_2(1)=3>2=N_2(1)$. Hence, by Lemma \ref{comindex}, $2$ divides $i(K)$. \\
\textbf{A16:} $\n_2(a)=6$ and  $\n_2(b)\in \{7, 8\}$. In this case, we have $d(S_{12})=2$  and $e_{12}=4$ but, $\l_{12}=\frac{-3}{4}$ . Also, we have  $R_{\l_{12}}(F)(y)=(\psi(y))^2, $ where $\psi(y)=y-1$.   Unfortunately, we cannot apply  Theorem \ref{ore}. So, we analyze Newton polygons of second order.  Let   $\Phi_2(x)=x^4-\b$  with $\b \in \Z$ such that $\n_2(\b)=3$. By Lemma \ref{technical}, we choose this $\b$ such that $\n_2(\b^2+a) \ge 9$.  Consider the data $$T=(\ph_1(x), \l_{12}, \psi(y); \Phi_2(x)).$$ One can verify that  $T$ is a type of order 2. Let   $\n_2^{(2)}$ be the valuation of second order with respect to $T$.  The $\Phi_2$-adic development of $F(x)$ is
\begin{equation}\label{Devv(a)=6secondorder}
F(x)=x\Phi_2(x)^2+2 \b x\Phi_2(x)+(\b^2+a)+b. 
\end{equation}
Let $A_0(x)=(\b^2+a)+b. $
Applying Theorem 2.11 and Proposition 2.7  of   \cite{Nar}, we get
\begin{itemize}
	\item $\n_2^{(2)}(x)=3,$
	\item $\n_2^{(2)}(\Phi_2)=12,$
	\item $\n_2^{(2)}(x\Phi_2(x)^2)=27,$
	\item $ \n_2^{(2)}(16x\Phi_2(x))=31,$
	\item $ \n_2^{(2)}(A_0(x))=\min \{4\n_2(\b^2+a)+3, 4\n_2(b)\} \ge \min \{37, 4\n_2(b)\}.$
\end{itemize}
It follows that  $N_2(F)$  is the lower convex hull of the points $(0, \mu_0^{(2)}), (1, 31)$ and $(2, 27)$, where $\mu_0^{(2)}= \n_2^{(2)}(A_0(x))$. \\
  If $\n_2(b)=7$, then $\mu_0^{(2)}=4\n_2(b)=28$. Thus,   $N_2(F)=S_{11}^{(2)}$ has only one side of degree $1$ joining the points $(0, 28)$ and $(2, 27)$ with ramification index  $e_{11}^2=2$. Thus, $F(x)$ is regular with respect to $\n_2^{(2)}$. By Theorems 3.1 and 3.4 of \cite{Nar}, we see that $2\Z_K=\pF_{111} \cdot \pF_{121}^8,$  where  $\pF_{111}$ and $\pF_{121}$ are two prime ideals     of $\Z_K$ of residue degree $1$ each. Similarly, if $\n_2(b)=8$, we see that   $N_2(F)$ has again one side of degree $1$. By Lemma \ref{comindex}, $2$ does divides $i(K)$. \\
\textbf{A17:}  $\n_2(a)=6$ and  $\n_2(b)\ge 9$.  Let us  use  the same key polynomial $\Phi_2(x)$ as in the above case. In this case, we have   $\mu_0^{(2)} \ge 36$. It follows that  $N_2(F)=S_{11}^{(2)}+S_{12}^{(2)}$ has  two sides of degree $1$ each joining the points $(0, \mu_0^{(2)}), (1, 31)$ and $(2, 27)$. Their attached residual polynomials are separable as they are linear  polynomials of $\F^{(2)}[y]$. Applying Theorems 3.1 and 3.4 of \cite{Nar}, we see that $2\Z_K=\pF_{111} \cdot \pF_{121}^4\cdot \pF_{122}^4,$ where $\pF_{111}, \pF_{121}$ and  $\pF_{122}$ are three prime ideals of $\Z_K$ of residue degree $1$ each. Hence, by Lemma \ref{comindex}, $2$ divides $i(K)$. \\
\textbf{A18:}  $\n_2(a)=4$ and  $\n_2(b)\in \{5, 6\}$. In this case, we have  $d(S_{12})=4, e_{12}=2, \l_{12}=\frac{-1}{2}$.  Further, $R_{\l_{11}}(F)(y)=(y-1)^4 \in \F_{\ph_1}[y]$.  So, Ore's program does not able to factorize $2\Z_K$. So, let us pass to second order Newton polygons. Let $\Phi_2=x^2-\b$ with $\n_2(\b)=1$. Using  Lemma \ref{technical}, we choose  $\b$ such that $\n_2(\b^4+a) \ge 8$. Let $T=(\ph_1(x), \l_{11}, \psi(y); \Phi_2(x))$. One can verify that   the data $T$ is a type of order 2.  Let   $\n_2^{(2)}$ be the valuation of second order with respect to $T$.  The $\Phi_2$-adic development of $F(x)$ is
\begin{equation}\label{Devv(a)=4secondorder}
F(x)=x\Phi_2(x)^4+4\b x\Phi_2(x)^3+6x\b^2 \Phi_2(x)^2+4x \b \Phi_2(x)+(\b^4+a)x+b. 
\end{equation}
Let $A_0(x)=(\b^4+a)+b. $
Applying Theorem 2.11 and Proposition 2.7  of   \cite{Nar}, we get
\begin{itemize}
	\item $\n_2^{(2)}(x)=1,$
	\item $\n_2^{(2)}(\Phi_2)=2,$
	\item $\n_2^{(2)}(x\Phi_2(x)^4)=9,$
	\item $ \n_2^{(2)}(4x\b\Phi_2(x)^3)=13,$
	\item $ \n_2^{(2)}(6x\b^2\Phi_2(x)^2)=11,$
	\item $ \n_2^{(2)}(4x\b^3\Phi_2(x))=13,$
	\item $ \n_2^{(2)}(A_0(x))=\min \{2\n_2(\b^4+a)+1, 2\n_2(b)\} \ge \min \{17, 2 \n_2(b) \}.$
\end{itemize}
Let $\mu_0^{(2)}= \n_2^{(2)}(A_0(x))$. Then $N_2(F)$  is the lower convex hull of the points $(0, \mu_0^{(2)}), (1, 13), (2, 11)$ and $(4, 9)$. \\
 If  $\n_2(b)=5$, then  $\mu_0^{(2)}=10$.  It follows by $(\ref{Devv(a)=4secondorder})$ that $N_2(F)=S_{11}^{(2)}$ has only one side of degree $1$ joining the points $(0, 10)$ and $(4, 9)$. It ramification index $e_{11}^{(2)}=4$. Therefore, the form of the factorization of $2\Z_K$ is $[1, 1^8]$. By Lemma \ref{comindex}, $2$ does not divide $i(K)$. Similarly for $\n_2(b)=6$. \\
\textbf{A19:}  $\n_2(a)=4$ and  $\n_2(b)=7$. Let $\psi_1(x)=x-b$. Up to replace the lift of $\ph_1(x)$ by $\psi_1(x)$. Note that $\ol{\ph_1(x)}=\ol{\psi_1(x)}$ and $F(x) \equiv \ol{\psi_1(x)}^9$ in $\F_3[x]$.  Recall also that Ore's Theorem does not depend the monic irreducible liftings of the monic irreducible factors of $F(x)$ modulo the considered prime. Using Taylor expansion, the $\psi_1$-adic development of $F(x)$ is given by:
\begin{equation}\label{devpsiabeven}
F(x)=F(b)+F^{'}(b)\psi_1(x)+ \cdots  + \frac{F^{(8)}(b)}{8!}\psi_1(x)^8+\psi_1(x)^9.
\end{equation}
Further, we have $\n_2(F(b))=7$ and $ \n_2(\frac{F^{(k)}(b)}{k!}) \ge 7 $ for  $k=1, \ldots, 8$. It follows by $(\ref{devpsiabeven})$ that the polygon $N_{\psi_1}^{+}(F)=S_{11}$ has only one side of degree $1$ joining the points $(0, 7)$ and $(9, 0)$.  By Corollary \ref{ore}(2), $2\Z_K=\pF_{111}^9,$ where $\pF_{111}$ is a prime ideal of $\Z_K$ of residue degree $1$.\\
\textbf{A20:}  $\n_2(a)=4$ and  $\n_2(b) \ge 8$.  The polygon $N_{\ph_1}^{+}(F)$ is the same as in Case \textbf{A18}. Let us the key polynomial $\Phi_2(x)$ be as in Case \textbf{A18}. In this case     $\mu_0^{(2)}\ge 16$.  By $(\ref{Devv(a)=4secondorder})$, the polygon  $N_2(F)=S_{11}^{(2)}+S_{12}^{(2)}+S_{13}^{(2)}$ has three sides such that $d(S_{11}^{(2)})=d(S_{12}^{(2)})=1$ and $d(S_{13}^{(2)})=2$. The residual polynomials $R_{\l_{11}^2}^{(2)}(F)(y)$ and $R_{\l_{12}^2}^{(2)}(F)(y)$ as they are of degree $1$. Therefore,  $2\Z_K= \pF_{111}\cdot \pF_{121}^2\cdot\pF_{131}^2 \cdot \aF^2, $ where  $\pF_{111}, \pF_{121}$ and $\pF_{131}$ are three prime ideals of $\Z_K$ of residue degree $1$ each, and $\aF$ is a non-zero ideal of $\Z_K$ provided by the side $S_{13}^{(2)}$. By using the Fundamental Equality $(\ref{FE})$,  the factorization of $2\Z_K$ has  one of the following forms: $[1, 1^2, 1^2, 1^2, 1^2], [1, 1^2, 1^2, 1^4]$ or $[1, 1^2, 1^2, 2^2]$.  So, we have $L_2(1) \ge 3 > 2= N_2(1)$. Hence, by Lemma \ref{corollaryore}, $2$ divides $i(K)$. In particular, $K$ is not monogenic.
\end{proof}
Now, let us prove Theorem \ref{p=3}.
\begin{proof}[Proof of Theorem \ref{p=3}]\
\\ \textbf{C1:}	$a \equiv 1\md{3} $ and $b \equiv 1 \md{3}$. In this case, $\ol{F(x)}= \ol{\ph_1(x)\ph_2(x)\ph_3(x)}$ in $\F_3[x]$, where $\ph_1(x)=x-1$, $\ph_2(x)=x^4-x^3+x^2+1$ and $\ph_3(x)=x^4-x^3-x^2+x-1$. By Corollary \ref{corollaryore}(1), we get $3\Z_K=\pF_{111} \cdot \pF_{211} \cdot \pF_{311},$  with $f(\pF_{111}/3)=1$ and $f(\pF_{211}/3)=f(\pF_{311}/3)=4$. By Lemma \ref{comindex}, $3$ does not divide $i(K)$.\\
\textbf{C2:}	 $a \equiv 1\md{3} $ and $b \equiv -1 \md{3}$. In this case, $\ol{F(x)}= \ol{\ph_1(x)\ph_2(x)\ph_3(x)}$ in $\F_3[x]$, where $\ph_1(x)=x+1$, $\ph_2(x)=x^4+x^3+x^2+1$ and $\ph_3(x)=x^4+x^3-x^2-x-1$. By Corollary \ref{corollaryore}(1), as in  Case \textbf{C1}, the form of the factorization of $3\Z_K$ is $[1, 4, 4]$. Hence, $3$ is not a common index divisor of $K$.\\
\textbf{C3:} $a \equiv -1\md{3} $ and $b \equiv 1 \md{3}$.  In this case, $\ol{F(x)}= \ol{\ph_1(x)\ph_2(x)}$ in $\F_3[x]$, where $\ph_1(x)=x^3-x-1$ and  $\ph_2(x)=x^6+x^4+x^3+x^2-x-1$. Using Corollary \ref{corollaryore}(1), $3\Z_K= \pF_{111}\cdot \pF_{221}$ with $f(\pF_{111}/3)=3$ and $f(\pF_{211}/3)=6$. Hence, $3$ does not divide $i(K)$.\\
\textbf{C4:}   $a \equiv -1\md{3} $ and $b \equiv -1 \md{3}$. Similarly to the case \textbf{C3}, the form of the factor
ization of $3\Z_K$ is $[3, 6]$.\\
\textbf{C5:} $a \equiv 1\md{3} $ and $b \equiv 0 \md{3}$. In this case, $\ol{F(x)}= \ol{x(x^4-x^2-1)(x^4+x^2-1)}$ in $\F_3[x]$. Using Corollary \ref{corollaryore}(1), the form of the factorization of $3\Z_K$ is $[1, 4, 4]$.\\
\textbf{C6:} $a \equiv -1\md{3} $ and $b \equiv 0 \md{3}$. In this case, $$\ol{F(x)}= \ol{x(x-1)(x+1)(x^2+1)(x^2+x-1)(x^2-x-1)}$$ in $\F_3[x]$. Using Corollary \ref{corollaryore}(1), the form of the factorization of $3\Z_K$ is $[1, 1, 1, 2, 2, 2 ]$. Consequently, $3$ does not divide $i(K)$. \\
$\bullet$ In Cases \textbf{C7-C13}, $3$ divides both $a$ and $b$. Thus, $\ol{F(x)}= \ol{\ph_1(x)}$ in $\F_3[x]$, where $\ph_1(x)=x$. It follows that $N_{\ph_1}^{+}(F)$ is the lower convex hull of the points $(0, \n_3(b)), (1, \n_2(a))$ and $(9, 0)$. Note also that  by the hypothesis $(\ref{Hypothese})$, the conditions   $$\n_3(a) \ge  8 \,\, \mbox{and}\,\, \n_3(b) \ge 9$$ are not satisfied simultaneously. So, the conditions D7-D13 covers all cases when $3$ divides both $a$ and $b$.\\
\textbf{C7:}  $9 \n_3(a)<8 \n_3(b)$ and $\n_2(a) \in \{1, 3,  5, 7\}$. In this case, $N_{\ph_1}^{+}(F)=S_{11}+S_{12}$ has two sides of degree $1$ each, with respective  ramification indices $e_{11}=1$ and $e_{12}=8$. Using Corollary  \ref{corollaryore}(2), $3 \Z_K = \pF_{111}\cdot  \pF_{121}^8 $ with $f(\pF_{111}/3)=f(\pF_{121})=1$. \\
\textbf{C8:}   $9 \n_3(a)<8 \n_3(b),$ $\n_2(a) \in \{2, 6\}$ and $a_3 \equiv 1 \md{3}$.   In this case, $N_{\ph_1}^{+}(F)=S_{11}+S_{12}$ has two sides with respective degrees $d(S_{11})=1$ and $d(S_{12})=2$. Further, we have $R_{\l_{12}}(F)(y)= y^2+1$ which is irreducible in $\F_{\ph_1}[y] \simeq \F_3[y]$. So, Theorem \ref{ore} is applicable. Therefore, $3\Z_K= \pF_{111} \cdot  \pF_{121}^4$ with respective ramification indices $f(\pF_{111}/3)=1$ and $f(\pF_{121}/3)=2$.\\
\textbf{C9:} $9 \n_3(a)<8 \n_3(b),$ $\n_2(a) \in \{2, 6\}$ and $a_3 \equiv -1 \md{3}$. The polygon $N_{\ph_1}^{+}(F)$  is the same 	as in Case \textbf{ C8}.  But here, we have  $R_{\l_{12}}(F)(y)= y^2-1 =(y-1)(y+1)$ which is separable in in $\F_{\ph_1}[y]$. Thus, $F(x) $  is $3$-regular. By Theorem,  $3\Z_K= \pF_{111} \cdot  \pF_{121}^4 \cdot \pF_{122}^4$, where $\pF_{111}, \,   \pF_{121}$ and $ \pF_{122}$ are three prime ideals of $\Z_K$ of residue degree $1$ each.\\
\textbf{C10:} $\n_3(a)=4, \, \n_3(b) \ge 5$ and $b_3 \equiv 1 \md{3}$. In this case, 
$N_{\ph_1}^{+}(F)=S_{11}+S_{12}$ has two sides with respective degrees $d(S_{11})=1$ and $d(S_{12})=4$. We have also $e_{12}=2$ and $R_{\l_{12}}(F)(y)= y^4+1=(y^2+y-1)(y^2-y-1)$ which is separable in $\F_{\ph_1}[y]$. Thus, $F(x)$ is $3$-regular. By Theorem \ref{ore}, we obtain that  $3\Z_K= \pF_{111} \cdot  \pF_{121}^2 \cdot \pF_{122}^2$, where   $\pF_{121}$ and $ \pF_{122}$ are three prime ideals of $\Z_K$ of residue degree $2$ each, and  $\pF_{111}$ is a prime ideal of $\Z_K$ of residue degree $1$. \\
\textbf{C11:}  $\n_3(a)=4, \, \n_3(b) \ge 5$ and $b_3 \equiv -1 \md{3}$. The polygon $N_{\ph_1}^{+}(F)$  is the same 	as in Case \textbf{C10}.  Further, we have  $R_{\l_{12}}(F)(y)= y^4-1 =(y-1)(y+1)(y^2+1)$ which is separable in in $\F_{\ph_1}[y]$. So, $F(x) $ is $3$-regular. Therefore, $$3\Z_K = \pF_{111}\cdot \pF_{121}^2 \cdot \pF_{122}^2\cdot \pF_{123},$$ with ramification indices $f(\pF_{111}/3)=f(\pF_{121}/3)=f(\pF_{122}/3)=1$ and $f(\pF_{123}/3)=2$.\\
\textbf{C12:} $9 \n_3(a)>8 \n_3(b)$ and $\n_3(b) \in \{1, 2, 4, 5, 7, 8\}$. In this case, $N_{\ph_1}^{+}(F)=S_{11}$ has only one side of degree $1$ and ramification index equals $9$. By Corollary \ref{corollaryore}(2), we see that $3\Z_K = \pF_{111}^9$ where $\pF_{111}$ is a prime ideal of $\Z_K$ of residue degree $1$. \\
\textbf{C13:}    $9 \n_3(a)>8 \n_3(b)$ and $\n_3(b) \in \{3, 6\}$. In this case, $N_{\ph_1}^{+}(F)=S_{11}$ has only one side of degree $3$ and ramification index  $e_{11}=3$. Thus, $3\Z_K = \aF^3$, where $\aF$ is a non-zero ideal of $\Z_K$. According to the Fundamental Equality $(\ref{FE})$,  the form of factorization of $\aF$ belong to the set $\{ [1^3], [1, 1, 1], [1, 1^2], [1, 2]\}$. Thus, the factorization  of  $3\Z_K$  has one of the following forms:  $[1^9], [1^3, 1^3, 1^3], [1^3, 1^6]$ or $ [1^3, 2^3]$. Consequently, by Lemma \ref{comindex}, $3$ is not a common index divisor of $K$.\\
$\bullet$ In Cases \textbf{C14-C34}, we have $a \equiv 0 \md{3}$ and $b \equiv -1 \md{3}$. Thus,  $\ol{F(x)}=\ol{\ph_1(x)^9}$ in $\F_3[x]$, where $\ph_1(x)=x-1$. The $\ph_1$-adic development of $F(x)$ is 

\begin{equation}\label{devf1p3}
F(x)= 1+a+b+(9+a)\ph_1(x)+ \sum_{j=2}^{9} \dbinom{9}{j} \ph_1(x)^j.
\end{equation}
Let $\n=\n_3(1+a+b)$ and $\mu=\n_3(9+a)$. Note that $\n_3(\dbinom{9}{j})=2-\n_3(j)$ for every $j=2, \ldots , 9$. It follows, by the above $\ph_1$-adic development  of $F(x)$, that $N_{\ph_1}^{+}(F)$ is is the lower convex hull of the points $(0, \n), (1, \mu), (3,1)$ and $(9, 0)$. \\
\textbf{C14:} $(a, b) \in \{(0, 2), (0, 5), (3, -1), (3, 2), (6, -1), (6, 5)\}\md{9}$. In this case $\n=1$. By $(\ref{devf1p3})$, $N_{\ph_1}^{+}(F)$ has  only one side of degree $1$ joining the points $(0, 1) $  and $(9, 0)$ with ramification index $1$. By Corollary \ref{corollaryore}(2), $3\Z_K = \pF_{111}^9$, where $\pF_{111}$ is a prime ideal of $\Z_K$ of residue degree $1$.\\
\textbf{C15:} $(a, b) \in \{(3, 5), (6, 2)\}\md{9}$. In this case, $\n \ge 2$ and $\mu =1$. It follows that $N_{\ph_1}^{+}(F)=S_{11}+S_{12}$ has two sides of degree $1$ each joining the point $(0, \n), (1, 1)$ and $(9, 0)$. Their respective ramification indices equal $e_{11}=1$ and $e_{12}=8$. By Corollary \ref{corollaryore}(2), $3\Z_K=\pF_{111}\cdot \pF_{121}^8$, where $\pF_{111}$ and $ \pF_{121}$ are two prime of $\Z_K$ ideals of residue degree $1$ each.\\
\textbf{C16:}   $(a, b) \in \{(0, 8), (0, 18), (9, -1), (9, 8), (18, -1), (18, 17)\}\md{27}$. Here, we have  $\n=2 $ and $\mu \ge 2$. Thus,  $N_{\ph_1}^{+}(F)=S_{11}+S_{12}$ has two sides of degree $1$ each joining the point $(0, 2), (3, 1)$ and $(9, 0)$ with  $e_{11}=3$ and $e_{12}=6$. By Corollary \ref{corollaryore}(2), $3\Z_K=\pF_{111}^3\cdot \pF_{121}^6$, where $\pF_{111}$ and $ \pF_{121}$ are two prime  ideals of $\Z_K$ of residue degree $1$ each.\\
\textbf{C17:}   $(a, b) \in \{(0, -1), (9, 17)\}\md{27}$. Here, we have  $\n \ge 3 $ and $\mu = 2$. It follows that,  $N_{\ph_1}^{+}(F)=S_{11}+S_{12}+S_{13}$ has three sides of degree $1$ each joining the point $(0, \n), (1, 2) (3, 1)$ and $(9, 0)$ with  $e_{11}=1$,   $e_{12}=2$ and $e_{13}=6$. By Corollary \ref{corollaryore}(2), $3\Z_K=\pF_{111}^3 \cdot \pF_{121}^2 \cdot \pF_{131}$, where $\pF_{111}$, $ \pF_{121}$, and $\pF_{131}$ are three prime  ideals of $\Z_K$ of residue degree $1$ each.\\
\textbf{C18:}  $(a, b) \in  \{(18, 8), (18, 35), (45, 8), (45, 62), (72, 35), (72, 35)\}\md{81}$. In this case, $\n=3$ and $\mu\ge 3$. It follows by $(\ref{devf1p3})$ that $N_{\ph_1}^{+}(F)=S_{11}+S_{12}$ has two sides of degree $1$ each joining the point $(0, 3), (3, 1)$ and $(9, 0)$ with  $e_{11}=3$,   $e_{12}=6$. By Corollary \ref{corollaryore}(2), $3\Z_K=\pF_{111}^3\cdot \pF_{121}^6$, where $\pF_{111}$, $ \pF_{121}$ are two prime  ideals of $\Z_K$ of residue degree $1$ each.\\
\textbf{C19:} $(a, b) \in \{(18, 62), (99, 224), (180, 143)\}\md{243}$. In this case $\n=4$ and $\mu=3$. It follows that, $N_{\ph_1}^{+}(F)=S_{11}+S_{12}$ has two sides joining the points $(0, 4), (1, 3), (2, 2), (3, 1)$ and $(9, 0)$ with  respective ramification slopes  $\l_{11}=-1$ and $e_{12}=\frac{-1}{6}$ (see FIGURE 6). Further, we have $(1+a+b)_3\equiv 1 \md{3}$ and $(9+a)_3 \equiv 1 \md{3}$. It follows that  $R_{\l_{11}}(F)(y)=1+y+y^2+y^3= (y+1)(y^2+1)$ which is separable in $\F_{\ph_1}[y]$. So, $F(x)$ is $2$-regular. By Theorem \ref{ore}, we see that $3\Z_K= \pF_{111}\cdot\pF_{121}\cdot \pF_{131}^6$ with ramification indices equal $f(\pF_{111}/3)=f(\pF_{131}/3)=1$ and $ f(\pF_{121}/3)=2$.
	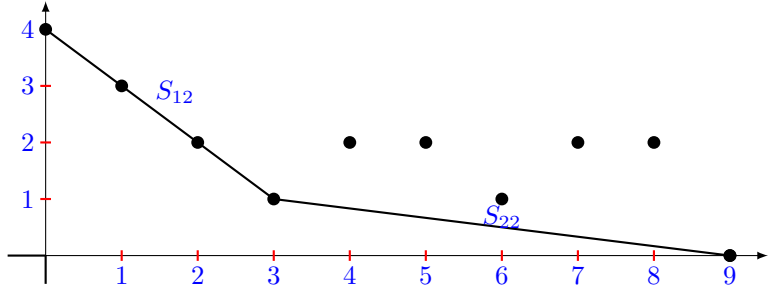
\begin{figure}[htbp] 
	\centering	
	\begin{tikzpicture}[x=1cm,y=0.75cm]
	\draw[latex-latex] (0,4.5) -- (0,0) -- (9.5,0);

	\draw[thick] (0,0) -- (-0.5,0);
	\draw[thick] (0,0) -- (0,-0.5);
	
	\draw[thick,red] (1,-2pt) -- (1,2pt);
	\draw[thick,red] (2,-2pt) -- (2,2pt);
	\draw[thick,red] (3,-2pt) -- (3,2pt);
	\draw[thick,red] (4,-2pt) -- (4,2pt);
	\draw[thick,red] (5,-2pt) -- (5,2pt);
	\draw[thick,red] (6,-2pt) -- (6,2pt);
	\draw[thick,red] (7,-2pt) -- (7,2pt);
	\draw[thick,red] (8,-2pt) -- (8,2pt);
	\draw[thick,red] (9,-2pt) -- (9,2pt);
	\draw[thick,red] (-2pt,1) -- (2pt,1);
	\draw[thick,red] (-2pt,2) -- (2pt,2);
		\draw[thick,red] (-2pt,3) -- (2pt,3);
	\draw[thick,red] (-2pt,4) -- (2pt,4);	
	\node at (1,0) [below ,blue]{\footnotesize  $1$};
	\node at (2,0) [below ,blue]{\footnotesize $2$};
	\node at (3,0) [below ,blue]{\footnotesize  $3$};
	\node at (4,0) [below ,blue]{\footnotesize  $4$};
	\node at (5,0) [below ,blue]{\footnotesize  $5$};
	\node at (6,0) [below ,blue]{\footnotesize  $6$};
	\node at (7,0) [below ,blue]{\footnotesize $7$};
	\node at (8,0) [below ,blue]{\footnotesize  $8$};
		\node at (9,0) [below ,blue]{\footnotesize  $9$};
	\node at (0,1) [left ,blue]{\footnotesize  $1$};
	\node at (0,3) [left ,blue]{\footnotesize  $3$};
	\node at (0,2) [left ,blue]{\footnotesize  $2$};
	\node at (0,4) [left ,blue]{\footnotesize  $4$};
	\draw[thick, mark = *] plot coordinates{(0,4) (1,3) (2,2) (3,1) (9,0)};
	\draw[thick, only marks, mark=*] plot coordinates{ (4,2) (5,2) (6,1) (7,2) (8,2) (9,0)};	
	\node at (1.7,2.5) [above  ,blue]{\footnotesize $S_{12}$};
	\node at (6,0.3) [above   ,blue]{\footnotesize $S_{22}$};
	\end{tikzpicture}
	\caption{   $N_{\ph_1}^{+}$ in Cases \textbf{C19-C22}}
\end{figure}\\
\textbf{C20:} $(a, b) \in \{(45, 35), (126, 197), (207, 116)\}\md{243}$.  Here,  $N_{\ph_1}^{+}$ is the same as  in Case \textbf{C19} (see FIGURE 6). Also, we have $(1+a+b)_3\equiv 1 \md{3}$ and $(9+a)_3 \equiv -1 \md{3}$. Thus, $R_{\l_{11}}(F)(y)=y^3+y^2-y+1$ which is irreducible over $\F_{\ph_1}$. So, Theorem \ref{ore} is applicable. Therefore, $$3\Z_K = \pF_{111}\cdot \pF_{121}^6,$$ where $f(\pF_{111}/3)=3$ and $f(\pF_{121}/3)=1$.\\
\textbf{C21:} $(a, b) \in \{(18, 143), (99, 62), (180, 224)\}\md{243}$. In this case,  $N_{\ph_1}^{+}$ is the same as in Case \textbf{C19} (see FIGURE 6). Further, we have $R_{\l_{11}}(F)(y)=y^3+y^2+y-1$, because $(1+a+b)_3\equiv -1 \md{3}$ and $(9+a)_3 \equiv 1 \md{3}$. So, $F(x)$ is $3$-regular. Hence, the form of the factorization of $3\Z_K$ is $[3, 1^6]$. \\
\textbf{C22:}   $(a, b) \in \{(45, 116), (126, 35), (207, 197)\}\md{243}$. Here,  $N_{\ph_1}^{+}$ is the same as in Case  \textbf{C19} (see FIGURE 6). Further, we have $(1+a+b)_3\equiv -1 \md{3}$ and $(9+a)_3 \equiv -1 \md{3}$.  Thus, $R_{\l_{11}}(F)(y)=y^3+y^2-y-1= (y-1)(y+1)^2$ which is not separable. So, $F(x)$ is not $\ph_1$-regular.  To place ourselves in the regular case, we  replace $\ph_1(x)=x-1$ by $\psi_1(x)=x+b$. Note that $\ol{\ph_1(x)}=\ol{\psi_1(x)}$ and $F(x) \equiv \ol{\psi_1(x)}^9$ in $\F_3[x]$.  Recall  that Ore's Theorem does not depend the monic irreducible liftings of the monic irreducible factors of $F(x)$ modulo the considered prime. Note also that $\n_3(a)= \n_3(1-b^8)=2$. On the other hand, since $\mu=3$, $a_3 \equiv 2  \md{9}$.  The $\psi_1$-adic development of $F(x)$ is 
\begin{equation}\label{devpsi1p=3}
F(x)=9b(-a_3+(1-b^8)_3)+9(a_3+b^8)+ \sum_{j=2}^{9} (-1)^{1+j} \dbinom{9}{j} b^{9-j} \psi_1(x)^j.
\end{equation}
One can check that $(1-b^8)_3 \equiv 5 \md{9}$. It follows that $\n_3(9b(-a_3+(1-b^8)_3))=2$. Thus,  $N_{\psi_1}^{+}(F)= S_{11}+S_{12}$ has two sides of degree $1$ each joining the points $(0, 2), (3, 1)$ and $(0, 9)$. Thus, $F(x)$ is $\psi_1$-regular. Hence it is $3$-regular. Applying Theorem \ref{ore},  the form of the factorization of $3\Z_K$ is $[1^3, 1^6]$. Therefore, $3$ is not a common index divisor of $K$.\\
\textbf{C23:}  $(a, b) \in\{(45, 197), (126, 116),  (207, 35)\}\md{243}$. In this case, $\n \ge 5$ and $\mu=3$. Thus,  $N_{\ph_1}^{+}=S_{11}+S_{12}+S_{13}$ has three sides joining the points $(0, \n), (1,3), (2, 2), (3, 1)$ and $(9, 0)$ , where $d(S_{11})=d(S_{13})=1, e_{11}=1, e_{13}=6$ and $d(S_{12})=2, e_{12}=1$. Further, we have $(9+a)_3=-1 \md{3}$. So,  $R_{\l_{12}}(F)(y)= -1+y+y^2$ which is irreducible over $\F_{\ph_1}$. By Theorem \ref{ore}, we see that $$3\Z_K=\pF_{111}\cdot \pF_{121}\cdot \pF_{131}^6,$$ where $f(\pF_{111}/3)=f(\pF_{131})=1$ and $f(\pF_{121}/3)=2$.\\
\textbf{C24:}  $(a, b) \in \{(18, 224), (99, 143)\} \md{243}$. Let $\psi_1(x)=x+a_3$. Using Binomial Theorem, the $\psi_1$-adic sevlopement of $F(x)$ is 
\begin{equation}\label{devpsi1p=3a_3}
F(x)=-a_3^9-aa_3+b+9a_3(1+a_3^7)+ \sum_{j=2}^{9} (-1)^{1+j} \dbinom{9}{j} a_3^{9-j} \psi_1(x)^j.
\end{equation} Here, we have $\n_3(-a_3^9-aa_3+b)=4$ and $\n_3(9a_3(1+a_3^7))=3$. It follows that $N_{\psi_1}^{+}=S_{11}+S_{12}$ has two sides joining the points $(0, 4), (1, 3), (2, 1), (3, 1)$ and $(9, 0)$. The  side $S_{11}$ has  degree $3$ and the side $S_{12} $ is of  degree $1$. On the other hand, we have $\frac{-a_3^9-aa_3+b}{3^4} \equiv \frac{9a_3(1+a_3^7)}{3^3} \equiv -1 \md{3}$. Thus, $R_{\l_{11}}(F)(y)=-1-y-y^2+y^3$ which is irreducible in $\F_{\psi_1}[y]$. So, $F(x)$ is $\psi_1$-regular. Then it is $3$-regular. By applying Theorem \ref{ore}, $3\Z_K= \pF_{111}\cdot\pF_{121}^6$, where $f(\pF_{111}/3)=3$ and $f(\pF_{121}/3)=1$.  By Lemma \ref{comindex}, $3$ does not divide $i(K)$.  \\
\textbf{C25:} $ a \equiv 18\md{243}$ and $b \equiv 62\md{243}$. Let $\psi_1(x)=x+a_3$ be as in Case \textbf{C24}. According to  the $\psi_1$-adic development $(\ref{devpsi1p=3a_3})$ of $F(x)$, we have $\n_3(-a_3^9-aa_3+b) \ge 5$ and $\n_3(9a_3(1+a_3^7))=3$. Thus, $N_{\psi_1}^{+}=S_{11}+S_{12}+S_{13}$ has three sides joining the points $(0, \n_3(-a_3^9-aa_3+b) ), (1, 3), (2, 2), (3, 1)$ and $(9, 0)$. Also, we have, $d(S_{11})=d(S_{13})=1, d(S_{12}°=2, \frac{9a_3(1+a_3^7)}{3^4} \equiv -1 \md{3}$ and $R_{\l_{12}}(F)(y)=-1-y+y^2$ which is separable. So, Theorem \ref{ore} is applicable. Therefore, the  form of the factorization of $3\Z_K$ is $[1, 2, 1^6]$. \\
	\textbf{C26:}    $(a, b) \in \{(72, 8), (234, 89), (153, 170)\}\md{243}$.  In this case, $\n=4$ and $\mu\ge 4$. It follows that $N_{\ph_1}^{+}=S_{11}+S_{12}$ has two sides joining the points $(0, 4), (2, 2), (3, 1)$ and $(9, 0)$, with $d(S_{11})=3, e_{11}=1, d(S_{12})=1,$ and $e_{12}=6$. Further, we have $R_{\l_{12}}(F)(y)= 1+y^2+y^3=(y-1)(y^2-y-1)$ which is separable  in $\F_{\ph_1}[y]$, since $(1+a+b)_3=1 \md{3}$. Thus, $F(x)$ is $3$-regular. By Theorem \ref{ore}, one has: $3\Z_K =\pF_{111}\cdot\pF_{112}\cdot\pF_{121}^6$ with  residue degrees $f(\pF_{111}/3)=f(\pF_{121}/3)=1$ and $f(\pF_{112}/3)=2$.\\
		\textbf{C27:}   $(a, b) \in \{(153, 8), (72, 89), (234, 170)\}\md{243}$. Here, $N_{\ph_1}^{+}$ is the same as in Case \textbf{C26}. Further, we have $(1+a+b)_3=-1 \md{3}$. Thus,  $R_{\l_{11}}(F)(y)=-1+y^2+y^3$ which is  irreducible  in $\F_{\ph_1}[y]$. Therefore, the form of the factorization of $3\Z_K$ is $[3, 1^6]$.\\
		\textbf{C28:}  $(a, b) \in \{(72, 170), (234, 8)\}\md{243}$. Let $\psi_1(x)=x+a_3$ be as in Case \textbf{C24}. In this case, we have  $\n_3(-a_3^9-aa_3+b) \ge 7$ and $\n_3(9a_3(1+a_3^7))=4$. By $(\ref{devpsi1p=3a_3})$,  $N_{\psi_1}^{+}=S_{11}+S_{12}+S_{13}+S_{14}$ has four  sides of degree $1$ each  joining the points $(0, \n_3(-a_3^9-aa_3+b) ), (1, 4), (2, 2), (3, 1)$ and $(9, 0)$. So, $F(x)$ is $3$-regular.  Using Theorem \ref{ore}, we see that $3\Z_K=\pF_{111}\cdot \pF_{121}\cdot \pF_{131}\cdot \pF_{141}^6, $ where $\pF_{111}, \pF_{121},  \pF_{131}$ and $  \pF_{141}$ are four prime ideals of $\Z_K$ of residue degree $1$ each. Then, we have $L_3(1)=4 > 3 =N_3(1)$. Consequently, by Lemma \ref{comindex}, $3$ divides $i(K)$. \\
	\textbf{C29:}    $ a \equiv 153 \md{243}, b \equiv 8 \md{243}, 2 \mu > \n +2$ and $\n$ is odd. In this case, $N_{\ph_1}^{+}(F)=S_{11}+S_{12}+S_{13}$  has three sides of degree $1$ each joining the points  $(0, \n), (2, 2), (3, 1)$ and $(9, 0)$ with respective slopes $\l_{11}=\frac{2-\n}{2}, \l_{12}=-1$ and $\l_{13}=\frac{-1}{6}$. By Corollary \ref{corollaryore}(1), we obtain that $3\Z_K = \pF_{111}\cdot \pF_{121}\cdot \pF_{131}^6, $ where $\pF_{111},  \pF_{121}$ and $ \pF_{131}$ are three prime ideals of $\Z_K$ of residue degree $1$ each. \\
	\textbf{C30:}    $a \equiv 153 \md{243}, b \equiv 8 \md{243}$ and $ 2 \mu < \n +2$. In this case, $N_{\ph_1}^{+}(F)=S_{11}+S_{12}+S_{13}+S_{14}$  has four sides of degree $1$ each joining the points  $(0, \n), (1, \mu), (2, 2), (3, 1)$ and $(9, 0)$ with respective ramification indices $e_{11}=e_{12}=e_{13}=1$ and $e_{14}=6$. Therefore, the   form of the factorization of $3\Z_K$ is $[1, 1, 1, 1^6]$. Here, we have $L_3(1)=4 > 3= N_3(1)$. Hence, by Lemma \ref{comindex}, $3$ divides $i(K)$. \\
	\textbf{C31:} $a \equiv 153 \md{243}, b \equiv 8 \md{243},  2 \mu > \n +2, \n$ is even,  and $(1+a+b)_3\equiv 1 \md{3}$. In this case, $N_{\ph_1}^{+}(F)=S_{11}+S_{12}+S_{13}$  has three sides   $(0, \n), (2, 2), (3, 1)$ and $(9, 0)$ such that $d(S_{11})=2, e_{11}=1,  d(S_{12})=d(S_{13})=1, e_{12}=1$ and $ e_{13}=6$. Further, we have $R_{\l_{11}}= 1+y^2$ which is separable in $\F_{\ph_1}[y]$. Using Theorem \ref{ore}, we see that $3\Z_K=\pF_{111}\cdot \pF_{121}\cdot \pF_{131}^6,$ with residue degrees  $f(\pF_{111}/3)=1$ and $f(\pF_{122}/3)=f(\pF_{131}/3)=1$.  \\
		\textbf{C32:} $a \equiv 153 \md{243}, b \equiv 8 \md{243},  2 \mu > \n +2, \n$ is even,  and $(1+a+b)_3\equiv -1 \md{3}$. Here, the polygon $N_{\ph_1}^{+}(F)$ is the same as in the case \textbf{C31}. But, $R_{\l_{11}}= -1+y^2=(y-1)(	y+1)$ which is also separable in $\F_{\ph_1}[y]$. By Using Theorem \ref{ore},  the form of the factorization of $3\Z_K$ is $[1, 1, 1, 1^6]$. In this case, $3$ is a prime common index divisor of $K$.\\
	\textbf{C33:} $a \equiv 153 \md{243}, b \equiv 8 \md{243},  2 \mu = \n +2$,  and  $(1+a+b)_3\equiv -1 \md{3}$. In this case, $N_{\ph_1}^{+}(F)=S_{11}+S_{12}+S_{13}$  has three sides   $(0, \n), (1, \mu), (2, 2), (3, 1)$ and $(9, 0)$ such that $d(S_{11})=2, e_{11}=1,  d(S_{12})=d(S_{13})=1, e_{12}=1$ and $ e_{13}=6$. Further, the residual polynomial  $R_{\l_{11}}(F)(y)= y^2-y-1$, because   $(a+9)_3=-1 \md{3}$. So,  $R_{\l_{11}}(F)(y)$ is separable in $\F_{\ph_1}[y]$.  Hence, $F(x)$ is $3$-regular Applying Theorem, the form of the factorization of $3\Z_K$ is $[2, 1, 1^6]$.\\
\textbf{C34:} $a \equiv 153 \md{243}, b \equiv 8 \md{243},  2 \mu = \n +2,$ and   $ (1+a+b)_3\equiv 1 \md{3}.$ This case is similar as the above case. \\
$\bullet$ In the rest of the cases;  from Case \textbf{C35}, we have $a \equiv 0\md{3}$ and $b\equiv 1 \md{9}$. Thus,  $\ol{F(x)}=\ol{\ph_1(x)^9}$ in $\F_3[x]$, where $\ph_1(x)=x+1$. The $\ph_1$-adic development of $F(x)$ is 
\begin{equation}\label{devf1p3b1}
F(x)= -1-a+b+(9+a)\ph_1(x)+ \sum_{j=2}^{9} (-1)^{j+1} \dbinom{9}{j} \ph_1(x)^j.
\end{equation}
Let $\omega=\n_3(-1-a+b)$ and $\mu=\n_3(9+a)$. Then, $N_{\ph_1}^{+}(F)$ is is the lower convex hull of the points $(0, \omega), (1, \mu), 	(2, 2),  (3,1)$ and $(9, 0)$. We proceed analogously as in  Cases \textbf{C14-C34}.  In Cases \textbf{C35-C49}, the polynomial $F(x)$ is $\ph_1$ regular. Then, we apply Theorem \ref{ore} to obtain the factorization of $3\Z_K$.  For Cases  \textbf{C50, C51}, in order to apply Theorem \ref{ore}, we choose $\psi_1(x)=x+b$,  and for Cases  \textbf{C52, C53, C54}, we choose $\psi_1(x)=x-a_3$ as a lifts of $\ol{\ph_1(x)}$. In these cases the polynomial $F(x)$ is $\psi_1$-regular. So, we are able to use  Ore's theorem.  We have collected all cases which satisfy the same polygons and residual polynomials. This completes the proof of the theorem.
\end{proof}
Now, we prove Theorem \ref{p5}.
\begin{proof}[Proof of Theorem \ref{p5}]
 By  \cite[Proposition 2.13]{Na}, for any $\eta \in \Z_K$, we have the following index formula: \begin{equation*}\label{indexdiscrininant}
\n_p(D (\eta)) =  2 \n_p(( \Z_K : \Z [\eta])) + \n_p(D_K),
\end{equation*}
where $D(\eta)$ is the discriminant of the minimal polynomial of $\eta$ and $D_K$ is the discriminant of $K$. By the definition $(\ref{Defi(K)})$ of $i(K)$ and the above relation,   if $p$ divides $i(K)$, then $p^2$  divides $\Delta(F)$. Recall that  the discriminant of $F(x)= x^9+ax+b$ equals
\begin{eqnarray}\label{discriminant}
\Delta(F)=3^{18}b^8+2^{24}a^9.
\end{eqnarray}
On the other hand, by the result of \.{Z}yli\'{n}ski  \cite{Zylinski}, if $p$ divides $i(K)$, then $p<9$;  see also  \cite{Engstrom}). Therefore, the condidate prime integers to be a common index divisor of $K$ are $2, 3, 5$ and $7$. So, it suffice to prove Theorem \ref{p5} for $p=5$ and $p=7$.    \\	
Let us show that $5$ does not divide $i(K)$. If $5$ divides $i(K)$, then  $5$ divides $\Delta(F)$. By  Formula $(\ref{discriminant}),$ we see that 
$$(a, b) \in \{(0, 0), (1, 1), (1, 2), (1, 3), (1,4)\} \md{5}.$$
 \begin{enumerate}[label=(\alph*)]
	\item If $a$ and $b$ are both divisible by $5$, then $\ol{F(x)}= \ol{\ph_1(x)}^9$  in $\F_5[x]$, where $\ph_1(x)=1$.  By Hypothesis $(\ref{Hypothese})$, two cases arise: \\
	\textbf{Case 1:} $\n_5(a) < \frac{8}{9} \n_5(b)$ and $\n_2(a) \in \{1,\ldots, 7\}$. In this case, $N_{\ph_1}^{+}(F)= S_{11}+S_{12}$ has two sides, where $d(S_{11})=1, e_{11}=1$ and $d(S_{12}) \in \{2, 4\}$.\\ If $\n_5(a) \in \{1, 3, 5, 7\}$, then $d(S_{11})=1$ and $e_{12}=1$   Thus,  for $k=1,2$, $R_{\l_{1k}}(F)(y)$ is irreducible as it is of degree $1$. Thus, $F(x)$ is $2$-regular. Applying Theorem \ref{ore}, we see that $2\Z_K= \pF_{111}\cdot \pF_{121}^8$ with $f(\pF_{111}/5)=f(\pF_{121}/5)=1$. \\ If $\n_5(a) \in \{2, 6\}$, then $d(S_{12})=2$ and $e_{12}=4$. It follows that  $5\Z_K= \pF_{111}\cdot \aF^4$ where $\aF$ is a non-zero ideal of $\Z_K$. By the Fundamental Equality $(\ref{FE})$, $L_5(1) \le 3$. So, by  Lemma \ref{comindex}, $5$ does not divide $i(K)$.\\ Similarly, if $\n_5(a)=4$, then $d(S_{12})=4$ and $e_{12}=2$. Thus,   $5\Z_K= \pF_{111}\cdot \aF^2,$ where $\aF$ is a non-zero ideal of $\Z_K$. So, $L_5(1) \le 5$. Hence, $5$ does not divide $i(K)$.\\
	\textbf{Case 2:} $\n_5(a) > \frac{8}{9} \n_5(b)$ and $\n_2(b) \in \{1,\ldots, 8\}$. In this case, $N_{\ph_1}^{+}(F)= S_{11}$ with $d(S_{11}) \in \{1,3\}$. If $\n_5(b) \in \{1, 2, 4, 5, 7, 8\}$, then $d(S_{11})=1$ and $e_{11}=9$. By Corollary \ref{corollaryore}(2), $5\Z_K= \pF_{111}^9$ with $f(\pF_{111}/5)=1$. If $\n_5(b) \in \{3, 6\}$, then $d(S_{11})=3$ and $e_{11}=3$. It follows that $5\Z_K = \aF^3$ for some non-zero ideal of $\Z_K$. By the Fundamental Equality, it is impossible to see that  $L_5(1) \ge 6$. So, $5$ does not divide $i(K)$. 
	\item  Now, assume $(a, b) \in \{ (1, 1), (1, 2), (1, 3), (1,4)\} \md{5}.$   Reducing $F(x)$ modulo $5$, we see that $\ol{F(x)}=\ol{\ph_1(x)^2\ph_2(x)}$ in $\F_5[x],$ with $\deg(\ph_1(x))=1$ and $\deg(\ph_2(x))=7$. By applying Theorem \ref{ore} and using the Fundamental Equality (\cite[Theorem 4.8.5]{Co}), the form of the factorization  $5\Z_K$ is    $[1, 1, 7], [1^2, 7]$ or $ [2, 7]\}.$  By Lemma \ref{comindex}, $5$ does not divide $i(K)$.	
\end{enumerate}
We conclude that in every case, $5$ is not a common index divisor of $K$.
We use the same idea to prove that $7$ does not divide $i(K)$. Consequently, the rational primes which can divide $i(K)$ are $2$ and $3$.
\end{proof}
To illustrate our results, we give some numerical examples. 
\begin{examples}
	Let $\th$ be a root of a monic irreducible polynomial  $F(x)=x^9+ax+b \in \Z[x]$ and  $K=\Q(\th)$.
	\begin{enumerate}
		\item If $F(x)= x^9+289x+34$, then it is irreducible 	over  $\Q$ as  it is  a $3$-Eisenstein polynomial.  By Case $A4$ of Theorem \ref{p=2}, $2$ divides $i(K)$. So, $K$ is not monogenic. 
			\item If $F(x)= x^9+28\times 33^k x+352 \times 3^h,$ where $h$ and $k$ are two positive integers, the it is irreducible over  $\Q$ as it is  a $11$-Eisenstein polynomial.  By Case $A15$ of Theorem \ref{p=2}, $2$ divides $i(K)$. Hence, $K$ is not monogenic. 
			\item Let $p \ge 5$ be a rational prime and $F(x)=x^9+p^kx+p,$ where $k$ is a positive integer.  Since $F(x)$ is an Eisenstein polynomial with respect $p$, then it is irreducible over the fields of rationals. By Theorems \ref{p=2}, \ref{p=3}, and \ref{p5}, the index of $K$ is trivial; $i(K)=1$.
	\end{enumerate}
\end{examples}

\end{document}